\documentclass[preprint]{imsart}

\RequirePackage[OT1]{fontenc}
\RequirePackage{amsthm,amsmath}
\RequirePackage[numbers]{natbib}
\RequirePackage[colorlinks,citecolor=blue,urlcolor=blue,pdfstartview=FitB]{hyperref}

\usepackage{tikz}
\usetikzlibrary{matrix}

\usepackage{changepage}
\usepackage{float}
\usepackage{amsmath,amssymb,epsfig}
\usepackage{graphics}
\usepackage{tabularx}
\usepackage{framed}
\usepackage[mathscr]{euscript}
\usepackage{imsart}

\usepackage{epstopdf}

\usepackage{graphicx}

\newcommand{\dsubseteq}{\mathrel{
\rotatebox[origin=c]{-45}{$\subseteq$}}
}

\newcommand{\ddsubseteq}{\mathrel{
\rotatebox[origin=c]{45}{$\supseteq$}}
}

\startlocaldefs
\numberwithin{equation}{section}
\theoremstyle{plain}

\long\def\comment#1{}

\newtheorem{theorem}{Theorem}

\theoremstyle{definition}
\newtheorem{definition}{Definition}

\newtheorem{remark}{Comment}[section]

\newcommand{\be}{\begin{eqnarray}}
\newcommand{\ee}{\end{eqnarray}}

\newcommand{\ba}{\begin{array}}
\newcommand{\ea}{\end{array}}
\newcommand{\bs}{\begin{align}\begin{split}\nonumber}
\newcommand{\bsnumber}{\begin{align}\begin{split}}
\newcommand{\es}{\end{split}\end{align}}

\renewcommand{\(}{\left(}
\renewcommand{\)}{\right)}

\renewcommand{\hat}{\widehat}

\newcommand{\B}{\mathcal{B}}

\newcommand{\Ep}{{\mathrm{E}}}

\renewcommand{\Pr}{{\mathrm{P}}}

\def\x{{x}}

\renewcommand{\hat}{\widehat}
\renewcommand{\leq}{\leqslant}
\renewcommand{\geq}{\geqslant}

\endlocaldefs

\begin{document}

\begin{frontmatter}
\title{Targeted Undersmoothing}
\runtitle{\ \ Targeted Undersmoothing}

\begin{aug}
\author{\fnms{Christian} \snm{Hansen}\ead[label=e1]{chansen1@chicagobooth.edu}}

\thankstext{t1}{First version:  August 2016.  This version is of  \today.  Sanjog Misra would like to thank the Kilts Center for
Marketing at the University of Chicago Booth School of Business and the
Neubauer Family Foundation for their financial support of this research. Christian Hansen would like to thank the National Science Foundation as well as The University of Chicago Booth School of Business for financial support of this research.  Damian Kozbur would like to thank The University of Z\"urich for financial support of this research.  In addition, the authors would like to thank Susan Athey, Alexandre Belloni, Victor Chernozhukov, Whitney Newey, as well as seminar participants at KU Leuven Post-Model Selection Colloquium (2016), The University of St. Gallen, Brigham Young University, Vanderbilt, University of Naples Federico II, Swiss Statistics Seminar (2017), Asian Meetings of the Econometric Society (2017) for helpful comments.}

\runauthor{C. Hansen, D. Kozbur and S. Misra \ \ }

\affiliation{Chicago Booth }

\address{The University of Chicago \\ Booth School of Business \\ 5807 S. Woodlawn, 
Chicago, IL 60637\\
\printead{e1}\\
}

\author{\fnms{Damian} \snm{Kozbur}\ead[label=e2]{damian.kozbur@econ.uzh.ch}}

\affiliation{University of Z\"urich }

\address{ University of Z\"urich \\ Department of Economics \\ Sch\"onberggasse 1, 
8001 Z\"urich\\
\printead{e2}\\
}

\author{\fnms{Sanjog} \snm{Misra}\ead[label=e3]{sanjog.misra@chicagobooth.edu}}

\affiliation{Chicago Booth }

\address{The University of Chicago \\ Booth School of Business \\ 5807 S. Woodlawn, 
Chicago, IL 60637\\
\printead{e3}\\
}

\end{aug}

\begin{abstract} This paper proposes a post-model selection inference procedure, called \textit{targeted undersmoothing}, designed to construct uniformly valid confidence sets for functionals of sparse high-dimensional models, including dense functionals that may depend on many or all elements of the high-dimensional parameter vector.  The confidence sets are based on an initially selected model and two additional models which enlarge the initial model.  By varying the enlargements of the initial model, one can also conduct sensitivity analysis of the strength of empirical conclusions to model selection mistakes in the initial model.  We apply the procedure in two empirical examples: estimating heterogeneous treatment effects in a job training program and estimating profitability from an estimated mailing strategy in a marketing campaign.  We also illustrate the procedure's performance through simulation experiments.

\smallskip


\end{abstract}


\end{frontmatter}


\section{Introduction}

Large, complex data sets, often described by the moniker \textit{big}, have opened new avenues for empirical work in economics and the social sciences. These data can be extremely rich in the sense that they contain information on a large number of variables for each observation. Such high-dimensional settings\footnote{Formally, a high-dimensional setting is an asymptotic frame for a sequence of statistical models where the number of unknown parameters grows at least as quickly as the sample size.} offer many opportunities for empirical researchers to analyze complex phenomena but pose practical and theoretical problems because of the presence of a large number of explanatory variables.

One of the challenges created by data with many available covariates is the specification of the statistical model.  With many available predictors, it is easy to specify a highly-complex model with many parameters to be estimated. Unfortunately, a statistical model with too many parameters is likely to overfit, resulting in both poor out-of-sample predictive performance and poor statistical inference about functionals that depend on the true parameters of the model.  For example, informative inference about parameters in a linear regression model is impossible if the number of explanatory variables is larger than the sample size if one is unwilling to impose additional model structure.    

Regularization - constraining the estimated model to avoid perfectly fitting the sample data - is therefore required for building a useful high-dimensional model.  \textit{Ad hoc} regularization by specifying a low-dimensional parametric model is commonly employed in empirical applications.  There are also a variety of formal regularization devices that provably control over-fitting and produce high-quality forecasts under sensible conditions.  However, regularization may also lead to regularization bias and \textit{underfitting} - fitting a model which misses important features of the phenomenon under study - which also results in poor predictive performance and invalid inference about population objects of interest.  For a systematic overview of high-dimensional methods and related issues, see \cite{elements:book}.

A popular regularizing structure in the statistics and econometrics literature is \textit{sparsity}; see, for a general reference,  \cite{BuhlmannGeer2011}.  Sparsity is a general term for an assumption which states that the true model depends only on a small subset of the unknown parameters.   An example is the sparse linear regression model which is characterized by having many covariates, most of which have zero coefficients.  A \textit{sparse estimator} is an estimator which returns a model in which only a small number of estimated parameters are nonzero.  There are a variety of sensible sparse estimators in the literature.  Leading examples are $\ell_1$-penalized methods such as the lasso estimator of \cite{FF:1993} and \cite{T1996}.\footnote{Alternatives to the lasso estimator with similar properties include the Dantzig selector (see \cite{CandesTao2007}), forward stepwise regression (see \cite{Wang:UltraHighForwardReg},  \cite{Zhang:GreedyLeastSquares}, \cite{Tropp:Greed}, \cite{Submodular:Spectral}, \cite{DK:FSEL} ), SCAD (see \cite{FanLi2001}), and many others.} Many $\ell_1$-penalized methods and related methods have been shown to have good estimation properties with i.i.d. data even when perfect variable selection is not feasible; see, e.g., \cite{CandesTao2007}, \cite{MY2007}, \cite{BickelRitovTsybakov2009}, \cite{horowitz:lasso}, and the references therein.  Results for $\ell_1$ methods beyond simple i.i.d. data structures\footnote{See, for instance, \cite{BellChenChernHans:nonGauss} and \cite{BCHK:Panel}.} also suggest that this type of regularization has fairly general applicability.  Lasso is also useful as an input into a \textit{post-model selection estimator} where statistical estimation is performed using a model selected through some statistical device; examples include \cite{BC-PostLASSO}, \cite{BellChenChernHans:nonGauss}, \cite{BCH-PLM}, \cite{BCHK:Panel}, \cite{BCFH:Policy}.  
This paper studies constructing inferential quantities, such as confidence intervals, for functionals of unknown model parameters in high-dimensional settings under sparsity assumptions.

The use of regularization is problematic for statistical inference and construction of confidence sets.   Confidence intervals for parameters in models which are estimated with a regularized estimator can have extremely distorted coverage probabilities if the regularization is not explicitly taken into account.  Heuristically, the problem with inference arises because the regularized model may not be the true model, e.g. there may be model selection mistakes, which introduces an additional source of uncertainty. The difficulty of performing inference following regularized estimation has been documented formally by \cite{MR2422862} and \cite{potscher} among others.  As a result, development of valid post-regularization inferential procedures is an important area of current research. 
 
A leading case for which positive results regarding construction of uniformly valid inferential statements after regularization are available is for inference about low-dimensional sets of pre-specified coefficients in sparse linear regression models.  Methods available in this setting include post-double selection, as in \cite{BCH-PLM},  or debiasing, as in \cite{vdGBRD:AsymptoticConfidenceSets} and \cite{ZhangZhang:CI}.  In each of these cases, the model of interest is given by
$$
y_i = x_i' \beta_0 + \varepsilon_i, \ \ s_0 = |support(\beta_0)| < n  
$$
where $i$ indexes observations, $n$ denotes sample size, $y_i$ is an outcome, $x_i$ are covariates, $\varepsilon_i$ are idiosyncratic disturbance terms, and $\beta_0$ is an unknown parameter to be estimated with $support(\beta_0) = S_0$ and $s_0 = |S_0|$ .   The goal in these papers is then to construct a confidence interval for the simple linear functional 
$$
a(\beta_0) = [\beta_0]_1,
$$ 
where $[ \ \cdot \ ]_1$ denotes the first component of a vector.\footnote{Approaches in this setting can easily be extended to accommodate the case where the object of interest is a known, small finite-dimensional subset of the full parameter vector.} Such inferential results have been extended to various settings, including panel data (see \cite{BCHK:Panel}), various nonparametric settings (see \cite{BCFH:Policy}, \cite{AddSep}), settings with generalized linear models (see \cite{Farrell20151}, \cite{BCFH:Policy}), and quantile regression (see \cite{BCFH:Policy}).  The ideas in the \cite{BCH-PLM} can also be generalized to estimation of parameters defined by moment conditions whenever appropriate sparsity conditions hold and Neyman orthogonalizations of the moment conditions are available; see for example \cite{Victor:Chris:Martin:ARE}, \cite{DML}, and references therein.  In addition, \cite{CCK:AOS:2013}, \cite{CCK:AOP:2017} and \cite{Buehlmann:etal:bootstrap:2016} describe how bootstrapping can be used in conjunction with some of the previously cited techniques.  These bootstrapping techniques also allow control of family-wise error rates for a large number of hypothesis tests.  It is worth noting that in all of these procedures, in addition to sparsity in the equation of interest,  additional assumptions regarding sparsity of the relationships between the covariates are required.

The purpose of this paper is to propose and analyze a simple post-model-selection inferential procedure, \textit{targeted undersmoothing}, which is applicable for inference about $\vartheta_0$ defined by a general class of functionals $$\vartheta_0 = a(\Pr_0),$$ under a single sparsity condition on the model of interest with data generating process $\Pr_0$.  Importantly, the class of functionals we consider may be dense, in the sense that they depend non-trivially on the entire high-dimensional parameter vector, may depend on the process generating the observations $(x_i,y_i)$, and may correspond to objects that are not $\sqrt{n}$ estimable. Examples of such functionals are (i) the conditional mean of $Y$ at a particular point $X = x_0$, $x_0'\beta_0$, in a linear model and (ii) a heterogeneous treatment effect for an individual given a high-dimensional vector of characteristics of that individual.  More generally, the approach we propose provides a procedure that may be used to obtain inferential statements about a large class of functionals that are of interest to economists such as marginal effects, elasticities, and counterfactual quantities of interest such as profits and welfare.

Our proposal is to form confidence sets for $\vartheta_0$ by starting with a typical confidence interval obtained from an initially selected model and then systematically enlarging the interval by perturbing the model to account for possible model selection mistakes.  More formally, our proposed confidence set is constructed as the union of standard statistical confidence sets based on the convex hull of $CI( \hat S^{up}) \cup CI( \hat S^{low})$, where $CI(S)$ denotes a confidence region for $\vartheta_0$ based on a model $S$ under the assumption that $S$ is the correct model.  $\hat S^{up}$ and $\hat S^{low}$ are in turn models selected from the data based on
\begin{itemize}
\item [1.] An initially selected model $\hat S^0$ chosen via a standard method targeting model fit to the data.
\item [2.] Two additionally selected models: an upper model $\hat S^{up} \supseteq \hat S^0$ and a lower model, $\hat S^{low} \supseteq \hat S^0$ chosen by respectively targeting worst-case upper and lower bounds on the functional of interest that can be achieved by small augmentations to the model $\hat S^0$.
 \end{itemize}
 
In practice, the initial model selection is performed with a standard high dimensional estimator like lasso.  The subsequent model selection steps depend on the functional of interest and target the behavior of that functional accommodating model selection mistakes made in the first step.  The subsequent steps are important since mistakes are inherent to all model selection procedures unless unrealistic conditions are imposed on the formal setting.\footnote{Such conditions include $\beta$-min conditions, which assert that nonzero unknown parameters must be bounded uniformly away from zero in absolute value.}  In this paper, when discussing model selection mistakes, we mean variables $j \in S_0$ such that $j \notin \hat S^0$.  Note that model selection mistakes are captured by the set $S_0 \setminus \hat S^0$.  We let $\hat s$ denote $\hat s = |\hat S^0|$ and $\delta_{\hat s}$ denote $\delta_{\hat s} = | S_0 \setminus \hat S^0|$.  We make the strong but important assumption that the researcher has a known upper bound, $\overline s$,  on the number of possible model selection mistakes, $\overline s \geq \delta_{\hat s}$.  

We note that the properties of $\delta_{\hat s}$ for a given model selection procedure like lasso may be difficult to calculate.   A second option for choosing $\bar s$ exists when a researcher is willing to assume a value for $s_0$ but is unwilling to make assumptions about $\delta_{\hat s}$.  In this case, a simple and valid choice for $\bar{s}$ is $\bar{s} = s_0$.  Note that by construction, $\delta_{\hat s} \leq s_0$ which immediately gives $\bar s \geq \delta_{\hat s}$.\footnote{In practice, a situation could easily arise where $\hat s > \bar{s}$ if $\bar s$ is taken to be a bound on $s_0$.  This situation can occur because typical bounds on the behavior of lasso imply that $\hat s \leq O(1)s_0$ and not necessarily that $\hat s \leq s_0$; see \cite{BickelRitovTsybakov2009} and other references on lasso cited above.  In light of this possibility, bounds on $\delta_{\hat s}$ may be more desirable in practice even though such bounds depend on random quantities.} 

When constructing $\hat S^{low}$ and $\hat S^{up}$ as above, the two conditions $|\hat S^{low} \setminus \hat S^0 | \leq \bar s$ and $|\hat S^{up} \setminus \hat S^0 | \leq \bar s$ are enforced.  Enforcing these conditions ensures that all involved selected sets are relatively sparse, which is important for good performance in practice, and that, in theory, the second round of selection is sufficient to capture any selection mistakes made in the first step and capture the true model.  

The name `targeted undersmoothing' is motivated by a useful, though informal, heuristic analogy between high-dimensional estimation and nonparametric estimation.  A key problem in nonparametric regression estimation is to choose a bandwidth (for kernel-based estimates) or a set of approximating functions (in series- or sieve-based methods).  Sufficiently small bandwidths and more flexible sets of approximating functions each lead to undersmoothing in estimating the target function in the sense that bias bias may be taken to be small relative to sampling variation.  Undersmoothing can thus be used to justify inference based on correctly-centered Gaussian approximations.  For a review, see \cite{li:racine:book}.  Choosing a bandwidth or set of approximating functions is not unlike choosing a penalty parameter in $\ell_1$-penalized regression where smaller values of the penalty parameter result in more complex models.  

Unfortunately, simply decreasing the penalty parameter in penalized estimation of a sparse high-dimensional model does not alleviate bias in the same way as decreasing a bandwidth in a traditional kernel problem due to the complexity of the model space inherent in high-dimensional problems.  Heuristically, moderate strength signals whose exclusion leads to bias are hard to pick out from among the many irrelevant variables; and as the penalty parameter is lowered beyond theoretically justified levels, it is likely that the first variables to enter the model will be irrelevant signals that happen to be moderately correlated to the outcome in the sample at hand.  In this case, the decrease of the penalty parameter does not alleviate bias by introducing variables with moderate, but non-zero, coefficients that were previously missed and simultaneously introduces a type of endogeneity bias as those irrelevant variables that are introduced are precisely those with the highest correlation to the noise within the current sample.  
Intuitively, the targeted undersmoothing approach addresses this problem by undersmoothing in those directions that seem to be most likely to account for bias by directly focusing on the functional of interest rather than model fit.

Our paper complements several interesting papers that look at similar problems.  The work in \cite{DML} develops general theory for a procedure for inference about a relatively low-dimensional set of prespecified target parameters when machine learning is used to estimate some features of the model under weak conditions.  \cite{Athey:Wager:2015:TE:RF} study asymptotically Gaussian inference for heterogeneous treatment effects using random forests, and the ideas of \cite{Athey:Wager:2015:TE:RF} are extended to other objects of interest in \cite{GradientForests}.  Relative to the present work, the formal results in \cite{Athey:Wager:2015:TE:RF} and \cite{GradientForests} are developed in settings with low-dimensional controls.  \cite{Athey:Imbens:2016} study estimation of heterogeneous treatment effects in conjunction with machine learning; see also \cite{Athey:Imbens:etal:TE:PandP}.  Inference in \cite{Athey:Imbens:2016} relies on tree-based methods and sample-splitting where part of the sample is used to learn the splitting rule for the tree and the other part of the sample is used to do inference for heterogeneous treatment effects conditional on the tree learned in the first subsample.  \cite{Athey:Imbens:Wager:ApproxResidRebalancing} perform residual rebalancing to estimate average treatment effects with high dimensional control variables when regression equations are given by sparse linear models under very weak restrictions on the propensity score model that include cases where the propensity score does not have a natural sparse representation.  \cite{Cai:Guo:qBalls} consider construction of confidence sets for dense functionals given by $a(\beta) = \| \beta \|_l$ for various $1 \leq l \leq \infty$.  Perhaps the most closely related current papers are \cite{ZhuBradic:linear} and \cite{Zhu:Bradic:general:high:dim:testing}.  Both \cite{ZhuBradic:linear} and \cite{Zhu:Bradic:general:high:dim:testing} construct hypothesis tests for objects similar to those considered in our paper via $\ell_1$-projections of coefficient estimates to the set of coefficients consistent with the null. \cite{ZhuBradic:linear} only considers linear functionals while \cite{Zhu:Bradic:general:high:dim:testing} considers general nonlinear functionals but imposes stronger sparsity conditions than those employed below.  We compare the performance of the tests in \cite{ZhuBradic:linear} to inference based on our proposed targeted undersmoothing procedure in the simulation section of this paper.


This paper also complements recent work in \textit{selective inference}.   Selective inference refers to inferential techniques for parameters $\beta_{0,S}$ which depends on a model $S$.  The goal is to approximate the sampling distribution of an estimated $\hat \beta _{\hat S} | \hat S$; i.e. to approximate the distribution of an estimator conditional on the selected model $\hat S$.  See, for instance, \cite{exact:post:sel:inference:lasso}, which carries out selective inference in the high-dimensional linear model in the case that $\hat S$ is chosen using lasso.  Selective inference is a sensible analytic tool for assessing uncertainty about model parameters when the selected model will be fixed and utilized for subsequent applications.  Targeted undersmoothing and selective inference are designed for different objectives.  Targeted undersmoothing aims to deliver inferential statements about objects of interest as defined in the population model rather than the values of these objects after conditioning on a selected model.

The need to specify $\bar s$ is a limitation of our proposed method.  However, this limitation is not unique to this paper.  Approaches to undersmoothing in the traditional nonparametric literature also rely on \textit{ad hoc} decisions about exactly what one means by sufficiently small bandwidth or sufficiently flexible set of approximating functions, for example.  With few exceptions, high-dimensional estimators perform well under sparsity assumptions, and perform poorly when sparsity fails.\footnote{See for instance, \cite{rjive}, which allows more instruments than observations but does not impose sparsity in the first stage.}  Furthermore, to the best of the authors' knowledge, there are currently no reliable tests for the violation of sparsity in the statistics or econometrics literature.     

Given the dependence of the proposed procedure to the \textit{ad hoc} choice of $\bar{s}$, we feel that the proposed approach will be most helpful when viewed through the lens of sensitivity analysis.  Specifically, one may look at how confidence regions for objects of interest change as one varies $\bar{s}$ over sensible values, for example, $\bar{s}  \in \{0,1,...,\bar{s}^*\}$.  Because the exercise starts with a model selected through a high-quality model selection procedure, setting $\bar{s} = 0$ corresponds to this procedure producing no model selection mistakes which happens in scenarios where oracle model selection is possible; see \cite{FanLi2001}, \cite{Zou2006} , \cite{BuneaTsybakovWegkamp2007b}.  As one then considers increasing $\bar{s}$, one is considering scenarios where the initial selector is allowed to have made increasingly many selection mistakes.  By looking at several values for $\bar{s}$, one thus gains insight into how sensitive conclusions are to the number of model selection mistakes made by the initial selector.  This approach is similar to applications of sensitivity analysis in treatment effects estimation where a variety of approaches to sensitivity analysis exist for gauging sensitivity of causal estimators to violations of underlying identifying assumptions; see, for example, \cite{rosenbaum:book} and \cite{manski:book} for textbook reviews of classic approaches.

Two examples give an illustration of the targeted undersmoothing procedure.  The first example studies heterogeneous treatment effects in the Job Trainings Partnership Act of 1982.  The second example studies expected profit from individually-targeted advertising strategies derived from estimates of heterogeneous treatment effects.   In the first example we find that under mild assumptions on the sparsity level, it is not possible to reject the null hypothesis that the individual-specific heterogeneous treatment effect is zero for most individuals.   However, we reject the null hypothesis of no heterogeneity fairly robustly, even though we cannot pin down individual effects reliably.  By contrast, in the advertising example, we see that the confidence intervals for the parameters we estimate are relatively robust to different assumptions about the true underlying sparsity level.  We find strong evidence suggesting heterogenous responses of individuals to direct mail advertising.  We also find strong evidence that strategic mailing to individuals based on their characteristics yields substantially higher profits than either of two simple fixed mailing strategies we consider.

Finally, the paper presents a simulation study.  The simulation design is motivated by the direct mailing marketing campaign example.  An interesting feature of the simulation study is that using $\bar s = 1$ is sufficient for producing correct coverage probabilities in almost all designs, even when $s_0 > 1$ and as large as 16.  We find that procedures which make use of model selection but rely on perfect model recovery may have seriously distorted coverage, confirming previous results in the literature.  

\section{Preliminaries: Rates of Convergence for Estimated Functionals of High-Dimensional Sparse Models}\label{sec: rates}

This section serves as a preliminary to the main proposed inferential procedure by formally deriving some simple convergence rates for estimators of various classes of functionals based on a model chosen with a formal model selection procedure.  These results verify that estimators of even dense functionals based on sparse, post-model-selection estimators may have favorable statistical properties, though they do not deliver a formal inferential procedure.  In Section \ref{sec: main}, we give a procedure for constructing confidence regions around the estimates described in the present section.

\subsection{Framework}
Throughout, we simply write $D, \beta_0, p, q, k, s_0, \Pr, \textbf{F}$, etc, excluding $n$ from the notation.  Operations throughout the analysis are performed for each $n$. In the asymptotic analysis, all objects should be understood to belong to sequences - $ \{D_n\}_{n=1}^\infty, \{\beta_{0,n}\}_{n=1}^\infty, \{p_n\}_{n=1}^\infty, \{s_{0,n}\}_{n=1}^{\infty}, \{ \Pr_n\}_{n=1}^\infty,\{ \textbf{F}_n\}_{n=1}^\infty$, etc - each indexed by $n$.

For a sample size $n$, consider a dataset
$$
D = (z_i)_{i=1}^n
$$
which is a random sample jointly distributed according to a distribution $\Pr_{0}$  supported on some subset $\textbf{F} \subseteq \mathbb R^{n \times q}$.
The random variables $z_i \in \mathbb R^{q}$ are the observations and are indexed by $i=1,...,n$ for sample size $n$.  
Recall the classical definition of a statistical model is a set $\mathscr P = \{ \Pr \}$ of distributions $\Pr$ on $\textbf{F}$.  The statistical model is well-specified if $\Pr_0 \in \mathscr P$.

Often times it is convenient to associate a parameter to the set $\mathscr P$.  Here, we consider an association $\beta \mapsto \mathscr P_{\beta}$, where $\beta \in \B \subseteq \mathbb R^p$ and $\mathscr P_\beta \subseteq \mathscr P$.  Therefore, each value of $\beta$ associates to a subset of the statistical model.  We assume that $\cup_{\beta \in \B} \mathscr P_\beta = \mathscr P$ and that $\beta_0 \mapsto \mathscr P_0 \ni \Pr_0$. 

When the context is clear and there is no chance for confusion, we abuse notation slightly.  In discussing probabilities of events $\textbf{G} \subseteq \textbf{F}$, we write $\Pr(D \in \textbf{G})$ to mean $\Pr_0(D \in \textbf{F})$.  I.e.  probabilities, unless otherwise noted, are always taken with respect to the measure $\Pr_0$ of the data generating process.  This reduces clutter in the presentation.

We are primarily interested in high-dimensional applications where $p$ is large compared to $n$ and thus assume sparsity:  we maintain that only a small subset of the components of $\beta_0$ are nonzero.  We set $S_0 = support(\beta_0)$ and we define $s_0 = | S_0 |$, the number of nonzero components of the vector $\beta_0$.\footnote{The setting and results in this paper can be extended to the case that $\beta_0$ can be decomposed into a sparse component and a small component, so that $\beta_0 = \beta_0^{(1)} + \beta_0^{(2)}$, $|support(\beta_0^{(1)})| \leq s_0$, $\| \beta_0^{(2)} \|_2 \rightarrow 0$.}    In this setting, it is natural to consider estimators of $\beta_0$ which are based on model selection.

\begin{definition} A model selection procedure is defined by a map $\textbf M: \textbf{F} \rightarrow 2^{\{1,...,p\}} .$ In addition, a model-based estimator is a map $\textbf{b}: 2^{\{1,...,p\}}\times \textbf F \rightarrow \mathbb R^p$ such that for $K \subset \{1,...,p\}$ and $D \in \textbf F$, $support(b(K,D))\subset K$.  The composition  $\textbf{b} \circ (\textbf{M}, \text{id}_\textbf{F}) $, where  $\text{id}_{\textbf{F}}$  is the identity, $\text{id}_{\textbf{F}}(x) = x$, defines a post-model selection estimator  $D \mapsto \hat \beta$.  
\end{definition}

It is convenient to define a notion of high dimensional convergence, which depends on $s_0$, $p$, and $n$.  Let $\hat \beta$ be any measurable estimator $\textbf{F} \rightarrow \mathbb R^p$.  Let $\hat S$ denote the support of $\hat \beta$, $\hat S = support(\hat \beta)$, and let $\hat s = | \hat S|$ denote the number of nonzero elements of $\hat \beta$.  We define $\| \cdot \|_2$ to be the Euclidean norm, and $\| \cdot \|_{2,n}$ to be the $n^{-1/2}$-normalized Euclidean norm on $\mathbb R^n$.  The following definition is not standard, but useful in our discussion.  

\

\begin{definition} The sequence $(\textbf b, \textbf M)$, or more generally $\hat \beta$, is high-dimensionally consistent over a class of sequences $\mathscr D = \{\mathscr P\}$ if $\hat s= O(1)s_{0} $ with probability $1-o(1)$ and
$\| \hat \beta - \beta_{0} \|_2 = O_\Pr\left( \sqrt{{s_{0} \log p}/{n}}\right)$, uniformly over $\mathscr D$.  We abbreviate this by writing $(\textbf b, \textbf M) \in \textbf{ U}(\mathscr D)$ or $\hat \beta \in \textbf{ U}(\mathscr D)$.
\end{definition}

Existence of estimators $\hat \beta \in \textbf U(\mathscr D)$ will be taken as a given high level condition.  Many such estimators have been proposed and analyzed in the literature;  see, for example, the textbook \cite{BuhlmannGeer2011} and references contained there.  Since our interest in this paper is on inference for functionals, we do not restate sets of low-level conditions for specific estimators for brevity.  Rather, we focus on understanding the extent to which sparse estimators that satisfy Definition 3 can be used to reliably estimate large classes of functionals of the unknown parameter and the observed data. 

The choice to consider only estimators featuring the $ \sqrt{{s_0 \log p}/{n}}$ rate comes at a slight loss of generality, in favor of being concrete.  Most standard high dimensional estimators will achieve the above rates.  In other cases, the arguments can be easily adapted.

The next two subsections discuss estimation and statistical inference for general post-model-selection estimation techniques.   Researchers are often interested in a functional of a statistical model.  In economics, common examples of functionals of interest are average treatment effects, heterogeneous treatment effects, demand elasticities, etc.  An advantage of post-model-selection estimators is that the same selected model can be used to estimate a wide range of functionals.

\subsection{Explicitly defined functionals}\label{explicit-functionals}

In this first example, we consider functionals $a: \mathbb R^q \times \mathbb R^p \rightarrow \mathbb R$ which may depend on $D=(z_i)_{i=1}^n$ and $\beta$.  We consider the entire collection $\{ \vartheta_{0,i} \}_{i=1}^n = \{ a(z_i,  \beta_0) \}_{i=1}^n$, and we will be interested in understanding how well  $a(z_i, \hat \beta)$ approximates $a(z_i, \beta_0)$ in the $\| \cdot \|_{2,n}$ norm.

Define the following notion of linearizable which will be useful in establishing the next theorem.

\begin{definition} \textit{Linearization of $a$.}  For each $z$, there is $da(z): \mathbb R^p\rightarrow \mathbb R$, linear, and $c_a(z) \in \mathbb R$ such that for every $\beta\in \mathbb R^p$ 
we have
$$| a(z, \beta) - a(z, \beta_0) - d a(z)'( \beta-\beta_0) |\leq c_a(z) \| \beta- \beta_0 \|_2.$$

\end{definition}

\noindent In addition, define $A_a$ by the matrix $$A_a=\frac{1}{n} \sum_{i=1}^n da(z_i)da(z_i)'$$
and, for a set $K \subset \{1,...,p\}$, set $\phi_{\max }(K)(A_a)$ to be the largest eigenvalue of the principal submatrix of $A_a$ corresponding to the index set $K$.

\

\begin{theorem}
Suppose $\hat \beta \in {\normalfont \textbf{U}}(\mathscr D)$. Suppose further that $a$ are in a sequence of functionals which satisfy {Definition 3}.  Then
$$\frac{\| a(z_i, \hat \beta) - a(z_i , \beta_0)\|_{2,n}}{ \|c_a(z_i) \|_{2,n} +  [\phi_{\max}(\hat S \cup S)(A_a)]^{1/2}}   =  O_\Pr \(  \sqrt{\frac{s_0 \log p}{n}} \ \).$$
\end{theorem}

\

\begin{proof}  $\| a(z_i,\beta_0) - a(z_i,\hat \beta)\|_{2,n} \leq \| da(z_i)'(\beta_0 - \hat \beta) \|_{2,n} + \| c_a(z_i) \| \beta_0 - \hat \beta \|_2 \|_{2,n}$.   The first term is bounded by $\| \beta_0 - \hat \beta \|_2 \phi_{\max}(\hat S \cup S)(A_a)$.  The second term is bounded by $ \|c_a(z_i)\|_{2,n} \|\hat \beta - \beta_0 \|_2$.  Noting that $\| \beta_0 - \hat \beta \|_2 = O_\Pr[(s_0 \log p /n)^{1/2}]$ completes the proof. \end{proof}

\

When $a(z_i,\beta)$ is uniformly linearizable in the sense that $\max_{i \leq n} c_a(z_i) = O_{\Pr}(1) $, and does not blow up over subsets $K$ in the sense that $\max_{|K| \leq Cs_0} \phi_{\max}(K)(A_a) = O_{\Pr}(1)$ for $C$ sufficiently large, then the convergence rates simplify to ${\| a(z_i, \hat \beta) - a(z_i , \beta_0)\|_{2,n}}   =  O_\Pr (  \sqrt{{s_0 \log p}/{n}} \ ).$

When $a(z_i,\beta) = z_i'\beta$, then $c_a(z_i) = 0$ and $A_a = \frac{1}{n} \sum_{i=1}^n z_i z_i'$. The quantities $\max_{|K| \leq Cs_0} \phi_{\max}(K)(A_a)$ are known as maximal sparse eigenvalues. Under mild conditions on $z_i$, (see \cite{BickelRitovTsybakov2009}, \cite{BellChenChernHans:nonGauss}), the relevant sparse eigenvalues can be bounded by $O_{\Pr}(1)$.  In this case, the convergence rate $O_\Pr (  \sqrt{{s_0 \log p}/{n}} \ )$ is attained from Theorem 1.  Another application relevant to the empirical examples below is of estimating heterogeneous treatment effects.  Suppose $\beta_0$ can be partitioned into $\beta_0 = ([\beta_0]_E, [\beta_0]_H)$ with the two components giving  individual characteristic effects and characteristic-by-treatment interaction effects.  Both empirical illustrations below have such structure. Then if $a(z_i,\beta) = [z_i]_H ' [\beta]_H$, a consequence of the above theorem is $\| [z_i]_{H}' [\hat \beta]_H - [z_i]_{H}' [\beta_0]_H \|_{2,n} = O_\Pr(\sqrt{s_0 \log p /n} \ )$.

\subsection{Implicitly defined functionals}

The next theorem considers a different class of functionals of the parameter $\beta$.  We express the target in the context of m-estimators, following e.g. \cite{newey:semiparametric:1994}  and \cite{Victor:Chris:Martin:ARE}.  We focus on estimation of $\vartheta_0 \in \mathbb R$. We assume that $\vartheta_0$ is defined as a solution to moment conditions given by a function $\psi(z,\vartheta, \beta)$, which takes values in $\mathbb R$.\footnote{Extension to the setting where $\vartheta_0$ and $\psi(z,\vartheta,\beta)$ are finite dimensional vectors with $\dim(\vartheta_0) \leq \dim(\psi(z,\vartheta,\beta))$ is trivial, but requires additional notation.}  Explicitly, we assume our parameters $(\vartheta_0,\beta_0)$ are defined as a solution to

$$\frac{1}{n} \sum_{i=1}^n \Ep \psi(z_i,\vartheta_0,\beta_0) =0. $$

One sensible estimator $\hat \vartheta$ is obtained by using a plug-in $\hat \beta$, calculated in a previous estimation step.  Then $\hat \vartheta$ is defined via the sample moment:

$$\hat \vartheta \in \underset {\vartheta \in \mathcal A} {\text{argmin} } \left \| \frac{1}{n} \sum_{i=1}^n  \psi (z_i, \vartheta,\hat \beta) \right \|_2 $$

\noindent for some compact set $\mathcal A \subseteq \mathbb R$ which does not depend on $n$ and which contains $\vartheta_0$.  In the development below, we simplify notation and write $$ m(\vartheta,\beta) = \frac{1}{n} \sum_{i=1}^n \Ep[\psi(z_i,\vartheta, \beta)],  \ \ \hat { m}(\vartheta,\beta) = \frac{1}{n} \sum_{i=1}^n \psi(z_i,\vartheta, \beta).$$

We impose regularity conditions on the functions $m(\vartheta, \beta)$ and $\hat m(\vartheta, \beta)$ below before giving the rates of convergence for $\hat \vartheta$ estimated according to the above method.  

\begin{definition} Define the following sets centered around $(\vartheta_0, \beta_0)$ relative to sequences $(t,v,k)=(t_n,v_n,k_n)$:
$$\mathcal A_t := \{ \vartheta:| \vartheta - \vartheta_0 | \leq t \}$$
$$\mathcal B_{v,k} := \{ \beta: \| \beta - \beta_0 \|_2 \leq v\sqrt{s_0 \log p /n} \} \cap  \{ \beta: | support( \beta - \beta_0 ) | \leq  ks_0 \}$$
\end{definition}

\begin{definition} Linearization of $ m$. For each $ \vartheta \in \mathcal A_t$ there is $c_{ m}( \vartheta) \in \mathbb R$ and $d   m(\vartheta): \mathbb R^p \rightarrow \mathbb R$, linear, such that for every $(\vartheta,\beta) \in \mathcal A_t \times \mathcal B_{v,k}$, we have

$$|  m( \vartheta, \beta_0) -  m(\vartheta, \beta) - d  m(\vartheta)'(  \beta_0-\beta) |\leq c_{ m}( \vartheta) \| \beta_0 - \beta \|_2^2$$
and $c_{m}(\vartheta) = O(1)$ uniformly over $\mathcal A_t.$  
\end{definition}

\begin{definition} Uniform Stochastic Equicontinuity.  We have the following bound uniformly over $\mathcal A_t \times \mathcal B_{v,k}$

$$ \| m(\vartheta, \beta) - \hat m(\vartheta, \beta) \|_2 =O_\Pr(n^{-1/2}).$$

\end{definition}

\begin{definition} Identifiability.  Let $\Gamma_0 =  \frac{\partial}{\partial \vartheta}    m(\vartheta_0, \beta_0).$  The parameter $\vartheta$ is identifiable if $\Gamma_0$ exists and $$2\left \|  m(\vartheta_0, \beta_0) \right \|_2 \geq \min \left (  \|\Gamma_0'(\vartheta_0 - \vartheta)\|_2 , \ \iota^{-1} \right ), \ \lambda_{\min}(\Gamma_0' \Gamma_0) \geq \iota^{-1}$$ for all $\vartheta \in \mathcal A_t$ for some sequence $\iota=O(1)$.
\end{definition}

High level conditions like those captured in Definitions 4-7 are routinely used in m-estimation problems and can be established under a variety of primitive conditions.  Definition 4 simply defines appropriate local neighborhoods to the true parameters $\vartheta_0$ and $\beta_0$ for use in Definitions 5 and 6. Definition 6 defines a linearization of the ``population'' objective function $m(\vartheta,\beta)$.  This is a relatively weak condition which importantly does not require that $\hat{m}(\vartheta,\beta)$ is smooth. Definition 7 provides a uniform law of large numbers.  This condition can also be shown under weaker stochastic equicontinuity conditions like those in \cite{newey:semiparametric:1994} with additional assumption on the data generating process (like independent observations). For example, if $\hat m(\vartheta, \beta)$ is smooth with probability 1, then $d \hat m(\vartheta)$ can be defined analogously to $d m(\vartheta)$ above.  In this case, the statement in the definition of stochastic equicontinuity given in Definition 6 follows under the following three conditions: (1) a classical stochastic equicontinuity assumption, $\| dm(\vartheta)'(\beta_0 -  \beta) - d \hat m(\vartheta)'(\beta_0 - \beta)\|_2 = o(n^{-1/2})$; (2) a condition on the quality of linearization where $c_m(\vartheta) = O(1)$ and $c_{\hat m}(\vartheta) = O(1)$ for $c_{\hat m}(\vartheta)$ defined analogously to $c_{m}(\vartheta)$; and (3) a uniform law of large numbers over $\mathcal A_t$ where $\|m(\vartheta, \beta_0) - \hat m(\vartheta, \beta_0) \|_2 = O_\Pr(n^{-1/2})$.  Definition 7 ensures that given knowledge of the data generating process, $\vartheta_0$ is uniquely defined.

Finally, let $[v]_j$ denote the $j^{\text{th}}$ component of a vector $v\in \mathbb R^p$. For a set $K \subset \{1,...,p\}$, let $[v]_K$ denote a vector with components $[v]_j, j\in K$.  

\

\begin{theorem}
 Consider $\hat \beta \in {\normalfont \textbf{U}}(\mathscr D)$.  Suppose the conditions on the sets $\mathcal A_t, \mathcal B_{v,k}$ given in Definition 4 are met with $\min(t,v,k) \rightarrow \infty$.  Suppose that $m$ satisfies Definitions 5-7.  Then for $\hat \vartheta$ defined above, 
$$|\hat \vartheta - \vartheta_0|  =  O_\Pr \left (\max_{K \subseteq \{1,...,p\}: |K| \leq  ks_0, \ \vartheta \in \mathcal A_t} \left  \| [d  {\normalfont  {m}} (\vartheta)]_K \right \|_{2}  \sqrt{\frac{s_0 \log p}{n}} \ \right )$$
\end{theorem}

\

\begin{proof}
  
Note, for $n$ sufficiently large, $\hat \vartheta \in \mathscr A_t$ since $\mathscr A_t \supseteq \mathscr A$.  Let $r$ be the rate given in the statement of the theorem.  By the identifiability assumption, we have that for any $\delta>0$, 

$$\Pr( \| \vartheta_0 - \hat \vartheta \|_2 > \delta) \leq \Pr \left ( \| m(\hat \vartheta,  \beta_0) \|_2 \geq  \frac{ \min( \sqrt \iota \delta, \iota )}{2} \right ).$$
It therefore suffices to show that  $\| m(\hat \vartheta,  \beta_0) \|_2 <  O_\Pr(r)$.  By the triangle inequality, we have that 
$$\|m(\hat \vartheta,  \beta_0) \|_2 \leq I_1 + I_2 + I_3$$
where we define $I_1:= \| m(\hat \vartheta, \beta_0) - m(\hat \vartheta , \hat \beta) \|_2$, 
$ I_2:= \|  m(\hat \vartheta, \hat \beta) - \hat m(\hat \vartheta ,\hat \beta) \|_2 $, and
$I_3:= \| \hat m(\hat \vartheta, \hat \beta ) \|_2 $.  $I_1$ is $O_\Pr(r)$ by linearity (applying the Cauchy-Schwarz inequality to the $(\hat \beta - \beta_0)$ term in the linearization.)  Second, $I_2$ is $O_\Pr(n^{-1/2})$ by the uniform law of large numbers that follows from the imposed conditions.
Finally, by construction of the estimator, we have $I_3 = \| \hat m(\hat \vartheta, \hat \beta) \|_2 \leq \| \hat m( \vartheta_0, \hat \beta )\|_2$.  Application of the uniform law of large numbers gives $\| \hat m(\vartheta_0, \hat \beta) \|_2 \leq O_{\Pr}(n^{-1/2}) +\| m(\vartheta_0, \hat \beta)\|_2$.  Application of linearization gives $\| m(\vartheta_0, \hat \beta)\|_2 =O_{\Pr}(r)$.

\end{proof}

When the functional $\vartheta$ of interest is linear in $\beta$, then $\vartheta = \xi' \beta$ for some $\xi \in \mathbb R^p$.  In this case, we can set $\psi(z, \vartheta, \beta) = \vartheta - \xi'\beta$.  This gives $dm(\vartheta) = \xi$, and $c_m(\vartheta) = 0$.  Furthermore, $ \left  \| [d  {\normalfont  {m}} (\vartheta)]_K \right \|_{2} = \| [\xi]_K \|_2 $.  In this sense, the size of the vector $\xi$ is directly related to the calculated rate of convergence.   Note that a point forecast in a linear model is an example of this case.

Specializing further to the case that $\vartheta= [\beta]_1$, note $\xi$ has only a single nonzero component.  In this case, $ \left  \| [d  {\normalfont  {m}} (\vartheta)]_K \right \|_{2} =1$ for every $K$ containing the element 1.  The corresponding rate of convergence is $\sqrt{s_0 \log p/n}$.  This rate is slower than the parametric rate of $1/\sqrt{n}$.  Note that under certain regularity conditions, like those described in \cite{BCH-PLM}, $[\beta_0]_1$ can be estimated at the parametric rate.  

Despite the slower rates of convergence in some situations, the estimates described above do have the desirable property of simplicity.  The simplicity becomes more desirable when $\vartheta_0$ is more complicated than a linear functional.  In the simulation section of this paper, we compare estimators of $[\beta_0]_1$ using both the plug-in estimate described above, as well as a procedure based on \cite{BCH-PLM} to quantify any potential loss in estimation quality in certain finite sample settings.

\section{Targeted Undersmoothing as an Inferential Procedure}\label{sec: main}
The previous sections show that many functionals of interest can be calculated accurately from a single estimated high dimensional model.  In this section, we consider inference for functionals $\vartheta_0 = a(\Pr_0)$.\footnote{In Section \ref{explicit-functionals}, we also considered an entire profile $\{\vartheta_{0,i}\}_{i=1}^n  = \{ a(z_i, \beta_0)\}_{i=1}^n$.  We note here that we will be able to construct pointwise confidence regions for $\vartheta_{0,i} = a(z_i,\beta_0)$.  Uniform confidence regions would require additional adjustment.}

We make the strong but important assumption that the researcher has a known upper bound, $\overline s$,  on the number of model selection mistakes, defined by $\delta_{\hat s} = |S_0 \setminus \hat S^0|$.  If the researcher has a prior assumption on $s_0$, but is unwilling to to make assumptions on $\delta_{\hat s}$, one may also take $\bar s = s_0$.  Formally, we assume $\bar s \geq \delta_{\hat s}$ with probability $1-o(1)$.  As earlier, we assume that for $K \subset \{1,...,p\}$, there is an estimator $\hat \beta(K)=\textbf{b}(K, D)$ which depends on $K$ and the data $D$.  In addition, assume we can construct for each $K$ with cardinality less than $\bar{s} + \hat s$, an observable random interval $[\ell_K, u_K]$ which will cover $\vartheta_0$ with a desired pre-specified frequency if $K = support(\beta_0) = S_0$.  In other words, we maintain that the true model is relatively low-dimensional and that, if told the exact form of the true model, we could construct valid inferential statements for the object of interest conditional on estimating the true model.

Given these assumptions, we can define the following inferential procedure:

\smallskip

\noindent \textit{Algorithm 1.  }{Targeted Undersmoothing.}

\noindent \textbf{Step 1.}  Select a model $\hat S^{0}$ by a fixed model selection procedure \textbf{M}. 

\noindent \textbf{Step 2.}  For each $K$, let $[\ell_K, u_K]$ be an associated random interval.  Select $$\hat S^{\text{low}} = \underset{K: \hat S^{\text{0} } \subseteq K \subseteq [p]: |K \setminus \hat S^0 | \leq \overline s}{\text{argmin}}\ell_K$$

$$\hat S^{\text{up}} = \underset{K: \hat S^{0} \subseteq K \subseteq [p]: |K \setminus \hat S^0 | \leq \overline s}{\text{argmax}} u_K$$

\noindent \textbf{Step 3.} Set $[\ell,u] = [\ell_{\hat S^{\text{low}}}, u_{\hat S^{\text{lup}}}]$

\

 {Algorithm 1 takes an initially selected model and then searches for deviations that include that model and add no more than $\bar{s}$ extra variables.  To choose how to add variables, we do not look at model fit but rather which deviation leads to the largest change in inferential statements about the parameter of interest.  In the case of a confidence interval, we do this separately for the upper and lower bound of the interval.  This formulation intuitively conservatively captures the worst-case impact of up to $\bar{s}$ model selection mistakes on inference for the target quantity.}
 Figure \ref{TUFull} gives a schematic representation of the model selection timeline corresponding to Algorithm 1.

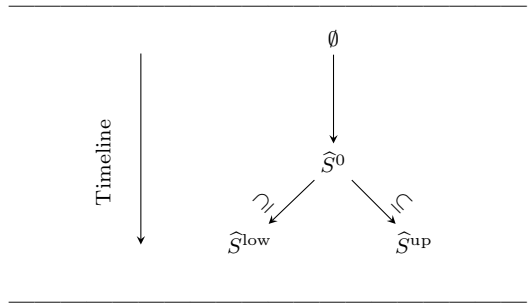
\begin{figure}[ht]
\caption{Targeted Undersmoothing: Schematic Diagram }
\center \textbf{------------------------------------------------------------}

\begin{tikzpicture}
\label{TUFull}
 \matrix (m) [matrix of math nodes, row sep=2em,
    column sep=1.5em]{
   \text{ }  & & \emptyset &  \\
   & & & \\
    & & \hat S^{0}  &   \\
   \hspace{1cm} & \hat S^{\text{low}}  & & \hat S^{\text{up}} \\};
  \path[-stealth]
    (m-1-3) 
            edge (m-3-3)
    (m-3-3) edge node [right ]  { \ $\dsubseteq$}  (m-4-4)
            edge node [left] {$\ddsubseteq$ \ }  (m-4-2)
    (m-1-1) edge node [left, rotate=90,yshift = 5mm, xshift = 5mm] {Timeline}  (m-4-1); 
\end{tikzpicture}
\center \textbf{------------------------------------------------------------}
\end{figure}
In order to give a formal result describing the properties of the targeted undersmoothing procedure, define the following simple condition:
 
 \
 
\begin{definition}
The intervals $[\ell_K,u_K]$, $K \in \mathscr K$, have uniform coverage probability $\alpha$ over $\mathscr K$ if 
$$\liminf_{n \rightarrow \infty} \inf_{K \in \mathscr K: S_0 \subseteq K} \inf_{\Pr_0 \in \mathscr P \in \mathscr D}  \Pr( \vartheta_0 \in [\ell_K, u_K]) \geq 1-\alpha.$$
\end{definition}

\begin{theorem}
\textit{Consider Algorithm 1.}    
Suppose that the intervals $[\ell_K, u_K]$ have uniform coverage probability $\alpha$ over $\mathscr K = \{ K: S_0 \subseteq K, |K \setminus \hat S^0| \leq \bar s\}$.   In addition, the sparsity bound $\overline s$ satisfies $\overline s \geq |S_0|$.   Then 
$$\liminf_{n \rightarrow \infty}  \inf_{\Pr_0 \in \mathscr P \in \mathscr D}  \Pr(\vartheta_0 \in [\ell, u]) \geq 1-\alpha.$$

\end{theorem}

\

\begin{proof}
The theorem follows from $ \Pr(\vartheta_0 \in [\ell, u]) \geq \Pr(\vartheta \in [\ell_{S_0 \cup \hat S^0}, u_{S_0 \cup \hat S^0}])$.  The right-hand side has $\liminf$ bounded by $1- \alpha$ by assumption.
\end{proof}

Note that when $\hat S^0$ is given by $\textbf{M}$ for some $(\textbf b, \textbf M) \in \textbf U(\mathscr D)$, then $\mathscr K$ can be taken as deterministic, using $\{ K: S_0 \subseteq K, |K| \leq O(1)s_0 + \bar s\}$, where the $O(1)$ term corresponds to the implied $\hat s \leq O(1)s_0$ bound in the definition of $\textbf{U}(\mathscr D)$.

The high-level assumption that  the intervals $[\ell_K, u_K]$ have uniform coverage probability $\alpha$ over $\mathscr K $ is stronger than the lone assumption that $[\ell_{S_0}, u_{S_0}]$ covers $\vartheta_0$ with probability $1-\alpha$.  Sufficient conditions guaranteeing uniform coverage probability $\alpha$ over $\mathscr K$ are easily stated for special cases like the high-dimensional linear model.  Such conditions are commonly employed in the econometrics literature (see for example \cite{BellChenChernHans:nonGauss}) and are characterized by (1) probabilistic lower and upper bounds on minimal and maximal sparse eigenvalues of the matrix $\frac{1}{n}\sum_{i=1}^nz_iz_i'$, (2) moment conditions on the covariates and residual terms, and (3) rate conditions on $\bar s$ and $p$.  Nevertheless, a result which uses a weaker notion than uniform coverage probability $\alpha$ over $\mathscr K$ could also be desirable.  

The main problem in deriving such a result under weaker conditions stems from the fact that $$\Pr( \vartheta_0 \in [\ell_{S}, u_S] | \ S \ \text{selected } ) \neq  \Pr( \vartheta_0 \in [\ell_{S}, u_S] ).$$  If $S$ is selected and $S$ contains some $j \notin S_0$, then $K \neq S_0$ for each $K \supseteq \hat S$. One way in which this issue can be addressed is if \textbf{M} has the further property that there exists a fixed set $ T \supseteq S_0$ such that $ \Pr ( {\normalfont \textbf{M}}  (D) \cap T^c \neq \emptyset ) = o(1)$ and $\Pr( \vartheta_0 \in [\ell_T, u_T])$ is bounded by $1-\alpha$ asymptotically.  If in addition, the sparsity bound $\overline s$ satisfies $\overline s \geq |T|$, then the statement of the theorem, $\liminf_{n \rightarrow \infty} \Pr(\vartheta_0 \in [\ell, u]) \geq 1-\alpha$, is recovered.
Informally, this condition states that the set of variables which are liable for being falsely selected into $\hat S^{0}$ can be controlled by $\bar s$.\footnote{We conjecture that in linear regression models under irrepresentability conditions on the design matrix, we may take $T = S_0$. However, since $\bar s$ is a user-specified tuning parameter in the first place, we do not follow this line of reasoning in this paper.}

Another procedure avoiding the assumption of uniform coverage probability $\alpha$ over $\mathscr K$ could be constructed by foregoing the initial model selection procedure, and taking $\hat S^0 = \emptyset$.  This would eliminate the problem. However, we note that taking $\hat S^0 = \emptyset$ will consider models which are in no sense local to the true model.  This implies that such a procedure could fail to have power against many fixed alternatives.  

An alternative to the above assumption is to adopt a sample splitting strategy.  We partition the set $\{1,...,n \}$ into a disjoint union $A \sqcup B$ of sets of equal (or approximately equal) size, uniformly at random.  We perform initial model selection on Sample A.  We calculate $\hat S^{\text{low}}, \hat S^{\text{up}}$ using only Sample B.  Formally, we outline the procedure here:

\

\noindent \textit{Algorithm 2.  }{Targeted Undersmoothing with Sample Split.}

\noindent \textbf{Step 0.}  Partition the sample $\{1,...,n\}$ into disjoint sets $A \sqcup B$. 
 
\noindent \textbf{Step 1.}  Select a model $\hat S^{0,A}$ by the model selection procedure $\textbf{M}(D_A)$ where $D_A$ is the data $D$ restricted to the subsample $A$. 

\noindent \textbf{Step 2.}  For each $K$, let $[\ell_K^B, u_K^B]$ be the associated random interval calculated using sample $B$.  Select $$\hat S^{\text{low}} = \underset{K: \hat S^{0,A} \subseteq K \subseteq [p]: |K \setminus \hat S^0 | \leq \overline s}{\text{argmin}}\ell_K^B$$

$$\hat S^{\text{up}} = \underset{K: \hat S^{0,A} \subseteq K \subseteq [p]: |K \setminus \hat S^0 | \leq \overline s }{\text{argmax}} u_K^B$$

\noindent \textbf{Step 3.} Set $[\ell,u] = [\ell_{\hat S^{\text{low}}}, u_{\hat S^{\text{lup}}}].$

\

\

\begin{figure}[ht]
\label{TUSplit}
\caption{Targeted Undersmoothing with Sample Split: Schematic Diagram}

\center  \textbf{------------------------------------------------------------}
\vspace{-1mm}

\begin{tikzpicture}
 \matrix (m) [matrix of math nodes, row sep=2.5em,
    column sep=.5em]{
     \text{ }   & \text{Sample A} & &  \text{Sample B}  &   \\
   \text{ }  & \emptyset & \text{ } & &  \\
   & & & \\
    &  \hat S^{0} &  &  \hat S^{0}  &   \\
   \hspace{1cm} & &   \hat S^{\text{low}}  & & \hat S^{\text{up}} \\};
  \path[-stealth]
    (m-2-2) 
            edge (m-4-2)
    (m-4-4) edge node [right ]  { \ $\dsubseteq$}  (m-5-5)
            edge node [left] {$\ddsubseteq$ \ }  (m-5-3)
    (m-1-1) edge node [left, rotate=90,yshift = 5mm, xshift = 5mm] {Timeline}  (m-5-1) 
    (m-4-2) edge [dashed] (m-4-4);
\end{tikzpicture}
\vspace{-3mm}
\center \textbf{------------------------------------------------------------}

\end{figure}
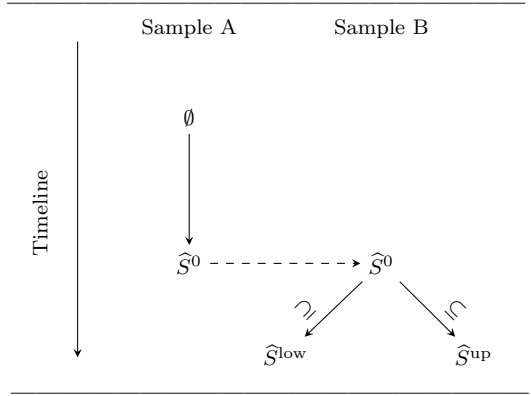

Using this procedure allows the uniform coverage probability assumption discussed above in Definition 8 to be dropped.  Instead, we adopt the following: \begin{definition}
The intervals $[\ell_K,u_K]$, $K \in \mathscr K$, have pointwise coverage probability $\alpha$ over $\mathscr K$ if for sequences $K\in \mathscr K$ such that $S_0 \subseteq K$, $$\liminf_{n \rightarrow \infty}  \inf_{\Pr_0 \in \mathscr P \in \mathscr D} \Pr( \vartheta_0 \in [\ell_K, u_K]) \geq 1-\alpha.$$
\end{definition}

\begin{theorem}
Consider Algorithm 2.  Suppose that the intervals $[\ell_K^B, u_K^B]$ have pointwise coverage probability $\alpha$ over $\mathscr K = \{ K: S_0 \subseteq K, |K \setminus \hat S^0| \leq \bar s\}$.    In addition, the sparsity bound $\overline s$ satisfies $\overline s \geq |S_0|$.    Then 
$$\liminf_{n \rightarrow \infty} \inf_{\Pr_0 \in \mathscr P \in \mathscr D} \Pr(\vartheta_0\in [\ell, u]) \geq 1-\alpha.$$
\end{theorem}

\begin{proof}
The theorem follows from $ \Pr(\vartheta_0 \in [\ell, u]) \geq \Pr(\vartheta_0 \in [\ell_{S_0 \cup \hat S^{0,A}}^B, u_{S_0 \cup \hat S^{0,A}}^B])$.  The right-hand side has $\liminf$ bounded by $1- \alpha$, using the fact that sample $A$ is independent of sample $B$.
\end{proof}

Algorithm 2 will in general produce wider confidence intervals, since it is it constrained to only work with sample $B$ for inference.    In our simulation study, we find that Algorithm 1 gives good coverage probabilities in all of the designs we tried.  

\begin{remark}
In addition to giving a procedure for constructing confidence sets, another use of targeted undersmoothing is for sensitivity analysis.  Theoretical properties of targeted undersmoothing depend on unknown - and to the best of our knowledge unlearnable - $\bar{s}$.  Rather than assuming $\bar s$ is known, trying several values $\bar s \in \{1,...,\bar s^*\}$ allows the researcher to see how sensitive confidence intervals and inference are sensitive to different values $s_0$.  We use this practice in the the empirical examples and the simulation exercises below.
\end{remark}

\begin{remark}
The above proposed algorithm is potentially computationally infeasible with even a moderate number $p$ of explanatory variables.  Therefore, in order to implement the procedure in practice, it may be necessary to approximate the quantities $[\ell,u]$.  

Depending on the exact nature of the problem, different approximations or bounds might be obtained with different methods.  For all of our simulation results and data applications in this paper, we add covariates indexed by $j$ into $\hat S^{(\text{low})}$ and $\hat  S^{(\text{up})}$ according to a simple greedy rule.  To be explicit, we perform the following algorithm:


\

\noindent \textit{Algorithm 3.  }{Greedy Approximation for $\hat S^{\text{low}}, \ \hat S^{\text{up}}$.}

\noindent \textbf{Initialize}: $\hat K^{\text{low}},  \hat K^{\text{up}} = \emptyset$

\noindent \textbf{While} $|\hat K^{\text{low}}|, | \hat K^{\text{up}}| \leq \overline s$

\hspace{6mm} \textbf{Set} $\hat j^{\text{low}}  = \text{arg} \min \ell_{\hat S^0  \cup \hat K^{\text{low}}   \cup \{j\} }$

\hspace{6mm} \textbf{Set} $\hat j^{\text{up}}   = \text{arg} \max u_{\hat S^0  \cup \hat K^{\text{up}}   \cup \{j\} }$

\hspace{6mm} \textbf{Set} $ \hat K^{\text{low}} =  \hat K^{\text{low}}  \cup \{\hat j\} $

\hspace{6mm} \textbf{Set} $\hat K^{\text{up}}   =  \hat K^{\text{up}}  \cup \{\hat j\} $

\noindent \textbf{End}

\noindent \textbf{Set} $\hat S^{\text{low}} = \hat S^0  \cup \hat K^{\text{low}}$

\noindent \textbf{Set} $\hat S^{\text{up}} = \hat S^0  \cup \hat K^{\text{up}}$

\

We note that other approximations to $[\ell,u]$ are also possible.  For example, semidefinite relaxations can give relatively quickly computable, valid lower bounds on $\ell$ and upper bounds on $u$ in some cases.  One could also adopt other solution techniques for obtaining approximate solutions to nonlinear integer programming problems.  Further exploration of these options may be useful, though we found the simple greedy algorithm presented above to perform well relative to other options in initial simulations.
\end{remark}

\begin{remark}
It is worth noting that targeted undersmoothing can also be used to carry out hypothesis testing.  This follows directly from the fact that confidence intervals can be constructed from inverted test statistics and vice versa.  Suppose the hypothesis of interest is $H_0: \vartheta_0 = \bar \vartheta$ for a prespecified value $\bar \vartheta$.  Suppose, given a model $S\subseteq \{1,...,p\}$, that $\hat W_S$ is an observable test statistic and that $\hat W_S$ corresponds to a p-value $\hat p_S$.  Then targeted undersmoothing can be used by choosing $\hat S = \hat S^0 \cup \hat K$ and by taking the set $|\hat K | \leq \bar s$  which makes the test most conservative (equivalently maximizing $\hat p_{\hat S}$.)
\end{remark}

\section{Empirical Examples}

In this section, we illustrate the use of targeted undersmoothing in two examples.  First, we study effects of job training programs on wages.  We are interested in estimating heterogeneous treatment effects in a setting where several individual characteristics are observed.  In the second example, we are interested in making individual-specific mailing strategies and estimating the profit gain from such a strategy.

\subsection{Application I:  Heterogeneous Treatment Effects from JPTA}

The impact of job training programs on the earnings of trainees, especially those with low income, is of interest to both policy makers and academic economists.   Evaluating heterogeneous causal effect of training programs on earnings is difficult due to the fact that individual characteristics vary across the sample; it is unlikely that many individuals share exactly the same values of observed covariates.  The problem is made worse the higher the dimension of the collected covariates.   

We consider data available from a randomized training experiment conducted under the Job Training Partnership Act (JTPA).  In the experiment, people were randomly assigned the offer of JTPA training services.  Given the random assignment of the offer of treatment, we focus this exercise on estimating the average treatment effect of the offer of treatment, or the intention to treat effect, conditional of individual characteristics.
In this example, we limit the analysis to the sample of adult males.    

To capture the effects of training on earnings, we estimate a model of the form 
$$
y_i =  x_i'  \beta_0  +  (d_i \cdot x_i)' \gamma_0 +\varepsilon_i
$$
where $d_i$ indicates whether training was offered, the outcomes $y_i$ are earnings, $x_i$ is a vector of covariates which includes a constant, $\varepsilon_i$ is an unobservable, and $( \beta_0,  \gamma_0)$ are parameters.  Earnings are measured as total earnings over the 30 month period following the assignment into the treatment or control group, and average earnings in the sample are \$19,147. Observed control variables are dummies for black and Hispanic persons, a dummy indicating high-school graduates and GED holders, five age-group dummies, a marital status dummy, a dummy indicating whether the applicant worked 12 or more weeks in the 12 months prior to the assignment, a dummy signifying that earnings data are from a second follow-up survey, and dummies for the recommended service strategy.  See \cite{Abadie:tal:JTPA} for detailed information regarding data collection procedures, sample selection criteria, and institutional details of the JTPA along with additional facts and discussion about the JTPA training experiment. 
In all, the dataset has 5102 observations.  

In this example, we are interested in estimating confidence intervals for individual specific treatment effects.  We form estimates by first calculating the post-lasso estimator of the coefficients
$$(\hat \beta_{PL}, \hat \gamma_{PL})$$
using the procedure described in the Implementation Appendix. Then for each individual $i$, we calculate the individual-specific intent to treat effect given by $$x_i'\hat \gamma_{PL}.$$

There are many ways to construct regressors from the set of dummy variables available.  In this example, we consider two methods to generate regressors.  The first method is based on common practice in econometrics of generating interactions.  The second method is based on the Hadamard-Walsh expansion\footnote{Details about this expansion as well as some of its advantages are described in \cite{set:fourier:sparse}} of the indicator variables described below, which generates a far larger set of regressors.  A larger set of regressors has advantages in that it can make any sparsity assumptions more plausible, though the resulting analysis may suffer in terms of statistical precision due to the increased complexity of the underlying model space.

To obtain the first construction we use for $x_i$, we consider all products of the discrete variables available.  That is, we adopt the common convention of including the dummy variables themselves, all first order interactions between the main dummy variables, all second order interactions, and all further higher order interactions.  Excluding empty and small cells, the dimension of the covariate space is 313.\footnote{Specifically, we start by eliminating all variables with $\leq 5$ nonzero entries in \textit{either} the control or treated subsample.  After these deletions, we then remove any variables if the corresponding diagonal R term in QR decomposition of the design matrix was $<10^{-6}$ over  \textit{either} the control or treated subsample.}  Therefore, with the treatment variable and constant, the total number of unknown parameters is 628.  Though the number of observations is larger than the sample size, the number of parameters is large enough that regularized estimation would be extremely helpful in terms of obtaining informative inference about model parameters.

Figure \ref{JTPACATE} presents pointwise confidence intervals for the individual specific effects for all individuals.\footnote{In principle, other descriptions of the treatment effect distribution can also be reported.  For instance, uniform bands for the sorted effects function could be obtained by combining the results in \cite{Victor:SortedTE} with targeted undersmoothing.  We choose to present pointwise confidence intervals for simplicity.}   The intervals are calculated using four methods.  The first panel presents estimates which use the entire set of control variables.  The second panel presents oracle-style confidence intervals based on post-lasso which ignore first stage model selection.  The third panel presents targeted undersmoothing estimates using $\bar s = 1$.  The fourth panel presents targeted undersmoothing estimates using $\bar s = 5$.  The targeted undersmoothing intervals are calculated with the forward selection greedy approximation described in the Section \ref{sec: main}.  In each case, we use Algorithm 1.

The figure shows that resulting confidence interval lengths using OLS estimates are quite large.  The interval lengths using the oracle-style confidence intervals are comparatively very tight.  Though the oracle-style intervals are expected to have poor performance in finite samples.   Using $\bar s=1$ we see that many of the interval lengths increase by nearly an order of magnitude.  Though interestingly, there is wide variation across individuals in terms of how much the corresponding confidence interval grows.  With $\bar s = 5$, we see that the intervals are in some cases nearly as large as with the OLS-based intervals. For most individuals, the corresponding intervals contain zero.
\begin{figure}[ht]
\caption{JTPA CATE Estimates: Interaction Specification}
\label{JTPACATE}
\includegraphics[scale=.130]{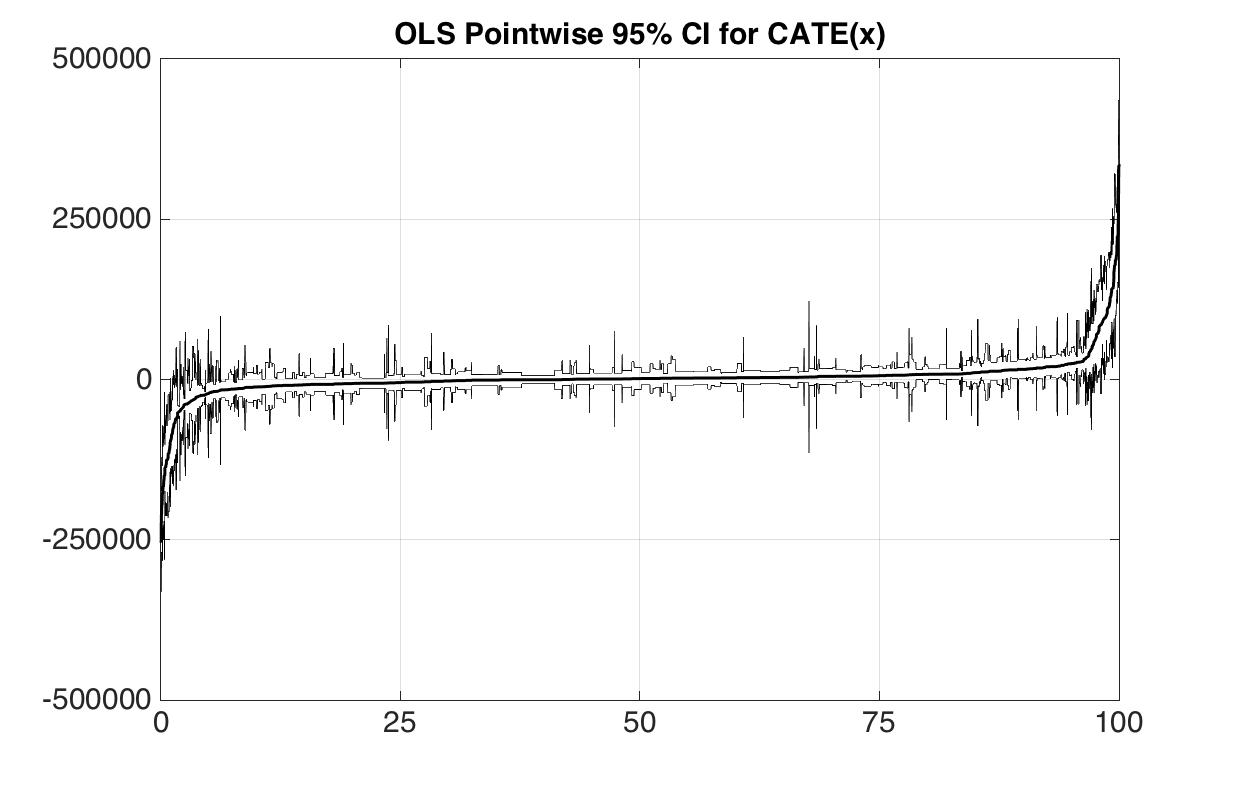}
\includegraphics[scale=.130]{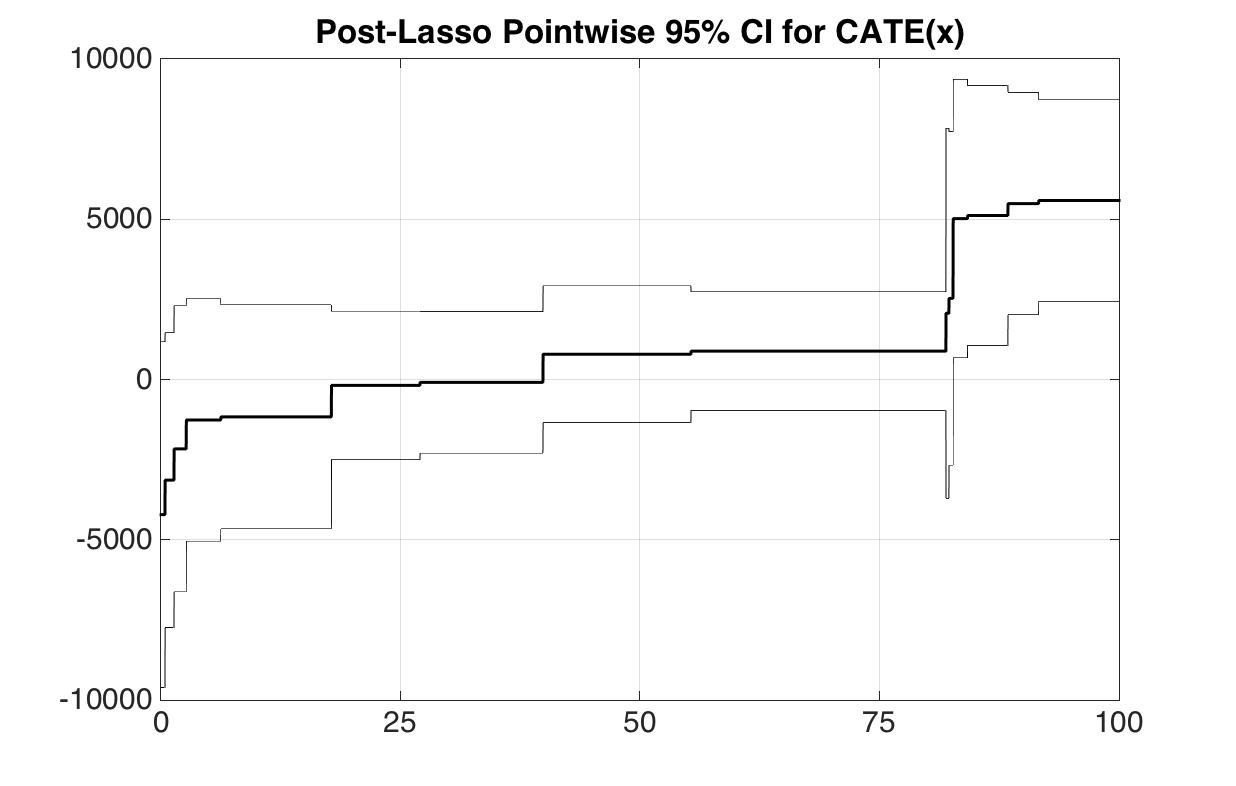}
\includegraphics[scale=.130]{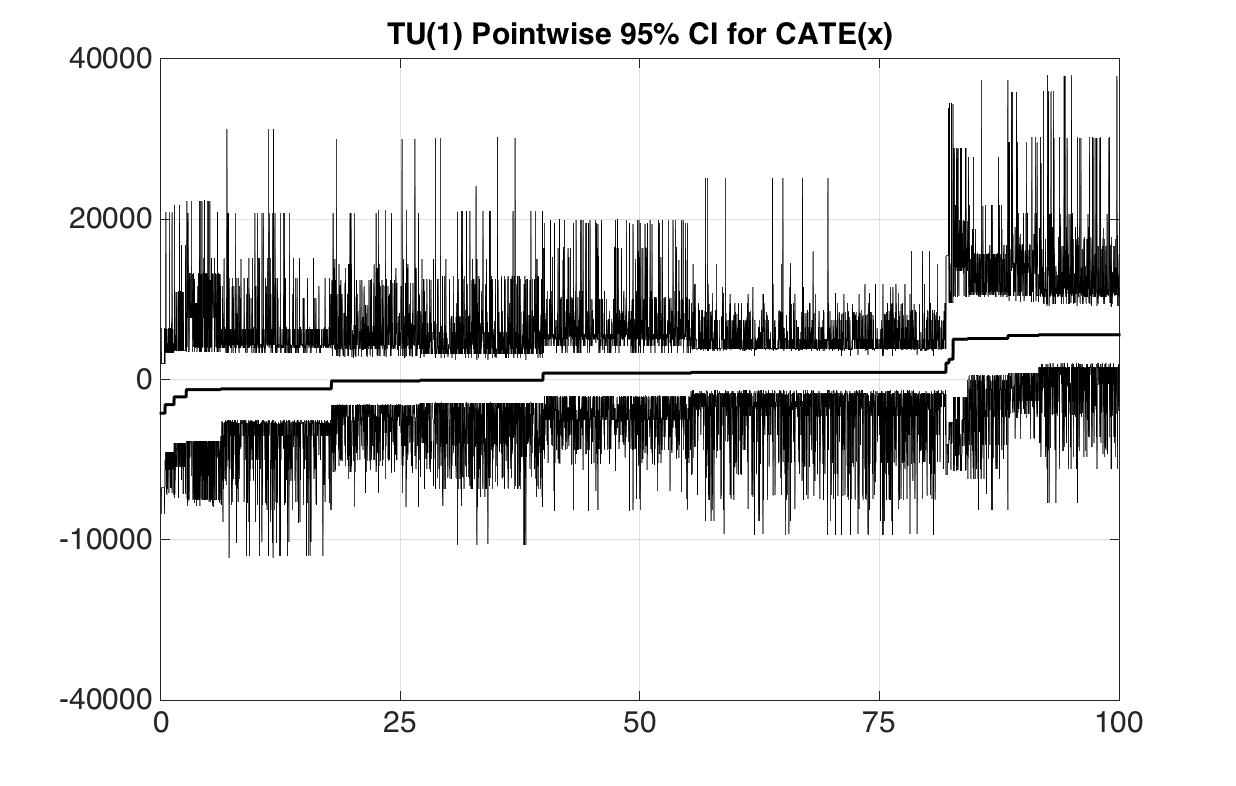}
\includegraphics[scale=.130]{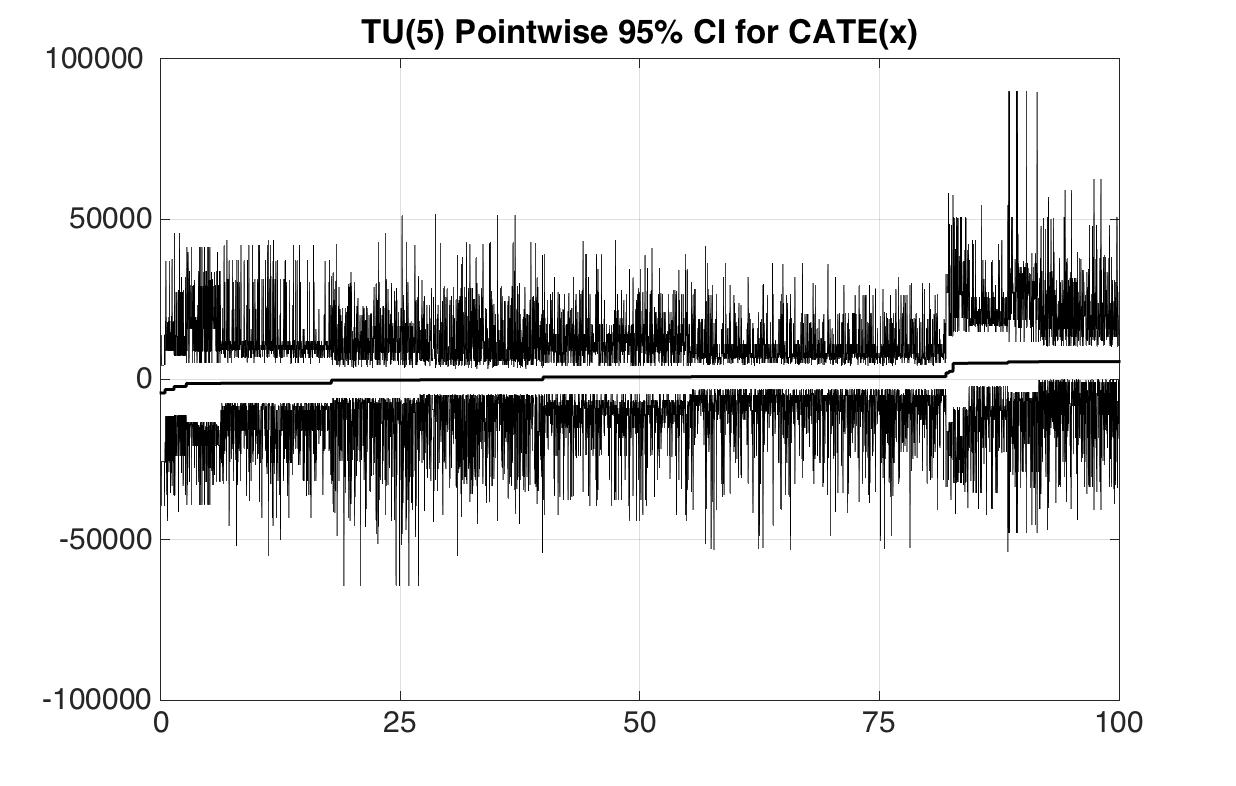}
\begin{flushleft}
\footnotesize{\textbf{Note:}  These figures report estimates of the treatment effect for each individual in the JTPA sample along with pointwise 95\% confidence intervals when the set of controls is constructed by taking all possible interactions of the baseline dummy variables.  Estimates based on OLS and Post-lasso are reported in the upper left and upper right panel respectively.  The lower left and lower right panels present results based on targeted undersmoothing with $\bar{s} = 1$ (``TU(1)'') and with $\bar{s} = 5$ (``TU(5)'') respectively.  It is important to note that vertical axis is different in each figure.}
\end{flushleft}
\end{figure}

Another testable hypothesis of interest is whether there is evidence of any effect heterogeneity.   Within the model, testing the null hypothesis of no treatment heterogeneity is equivalent to testing $H_0:  \gamma_0 = 0$.  As described in the previous section, a test can be implemented using the targeted undersmoothing procedure.   We implement this procedure using the standard Wald test.  The results are reported in Table \ref{JTPATest} for targeted undersmoothing using $\bar s \leq 10$.  We also report the corresponding Wald test using the entire vector of covariates (labeled OLS in the table), and an oracle-style Wald test (labeled PL in the table). We note that the OLS-based result is likely unreliable due to relying on a heteroskedasticity-consistent estimate of a large, full covariance matrix.  We reject the null hypothesis for $\bar s \leq 7$ at the 5\% level but fail to reject for larger $\bar{s}$.  An interesting property of the hypothesis testing scheme is that the degrees of freedom stay constant.  This means that the additional covariates entering the model correspond to components $x_i' \beta_0$, and not the interaction terms $(d_i \cdot x_i)' \gamma_0$.

\begin{table}[ht]
\caption{Testing the Null Hypothesis of No Treatment Effect Heterogeneity: Interaction Specification}
\label{JTPATest}

\centering{}%
\begin{tabular}{lcccc}
\hline 
Estimator & W-statistic & df &  p-value  \tabularnewline
\hline 
OLS	&	679.14 & 313 & 0.0000		\\
PL	&	17.1444 & 7	& 0.0088 	\\
TU(1)	&	16.4910	& 7 &    0.0210	\\
TU(2)	&	   15.9709		& 7 &      0.0254	\\
TU(3)	&	   15.5022	   & 7&      0.0301  \\
TU(4)	&	   15.0803		& 7 &     0.0350 \\
TU(5)	&	   14.7097   & 7& 	    0.0399\\
TU(6)	&	   14.4253	   & 7 & 	    0.0441\\
TU(7)	&		   14.1517   &7  &    0.0485 \\
TU(8)	&	   13.9339	   & 7&     0.0524	\\
TU(9) &	   13.5463	   & 7 &     0.0599	\\
TU(10)	&	   13.3584	   & 7 &    0.0638	\\
\hline 
\end{tabular}
\begin{adjustwidth}{2cm}{2cm}
\footnotesize{\textbf{Note:}  This table presents results for testing the null hypothesis of no treatment effect heterogeneity when the set of controls is constructed by taking all possible interactions of the baseline dummy variables.  We report the value of the Wald statistic (``W-statistic''), degrees of freedom (``df''), and associated p-value (``p-value'').  Results for testing this hypothesis based on OLS and Post-lasso estimates are provided in the first two rows of the table.  Rows labeled ``TU(j)'' correspond to targeted undersmoothing with $\bar{s} = j$.}
\end{adjustwidth}
\end{table}

The existence of a sparse representation of the regression function in the basis given by the interaction expansion is an important modeling assumption in the above analysis.  It is possible to perform a further robustness analysis by considering more expansive models.  In order to illustrate this point, we perform the analysis with an expanded set of transformations of the original dummy variables.  We consider the Hadamard-Walsh basis defined as follows.  Let $v_{i1},...,v_{ik}$ denote the original set of indicator variables. Let each subset $A \subseteq \{1,...,k\}$ index a transformation of $(v_{i1},...,v_{ik})$ given by $\psi_A(v_{i1},...,v_{ik}) = (-1)^{|A \cap \{j:v_{ij}=1\} | }$.  In the expanded model, we include regressors of the form $\psi_A(v_{i1},...,v_{ik})$.  In order to nest the previous analysis, we also include all of the interaction variables from the first specification.\footnote{We choose to only include $\psi_A$ terms as potential covariates for $1< |A| < 6$. Note that for $|A|=1$, the resulting transformations are perfectly correlated to the original indicator variables.}  The result is that $\dim(x_i) = 2927$, including the constant term.  After interacting $x_i$ with the indicator $d_i$, the total dimensionality of the model parameters is 5854, which exceeds the sample size $n=5102$.  

Figure \ref{JTPACATEHAD} presents pointwise confidence intervals for the individual specific effects for all individuals using the new, expanded set of transformations of the original variables.  In this analysis, OLS is no longer feasible because the dimensionality of the model exceeds the sample size.   The first panel presents oracle-style confidence intervals, which ignore first stage model selection.  The estimated distribution of heterogenous effects is much smoother than that obtained in Figure \ref{JTPACATE}.  Interestingly, the initial model selection selects terms from both the interaction expansion and the Hadamard-Walsh expansions.  The second panel presents targeted undersmoothing estimates using $\bar s = 1$, and the third panel presents targeted undersmoothing estimates using $\bar s = 5$. The targeted undersmoothing intervals are calculated with the forward selection greedy approximation described in the Section \ref{sec: main}.  As before, in each case, we use the single sample option described in Algorithm 1.
 
\begin{figure}[ht]
\caption{JTPA CATE Estimates: Hadamard-Walsh Specification}
\label{JTPACATEHAD}
\includegraphics[scale=.09]{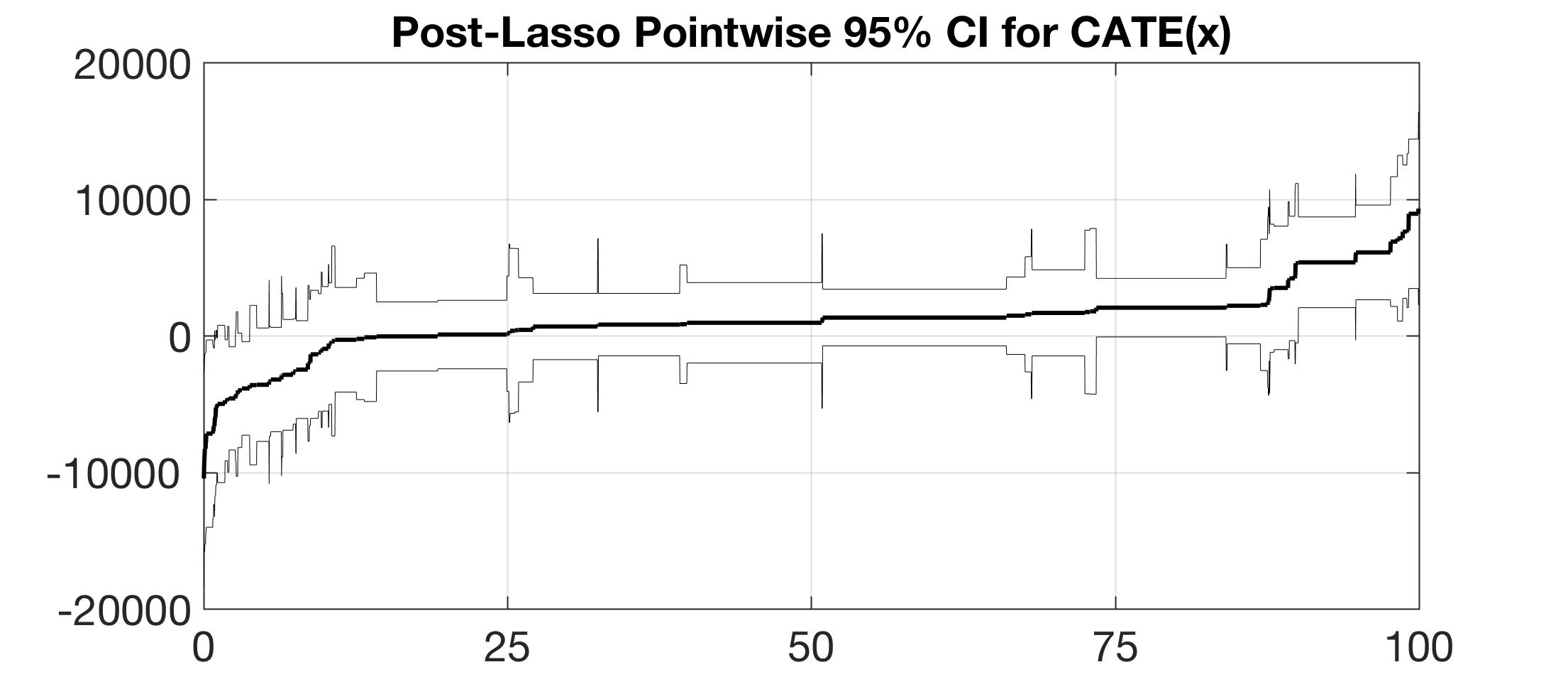}
\includegraphics[scale=.09]{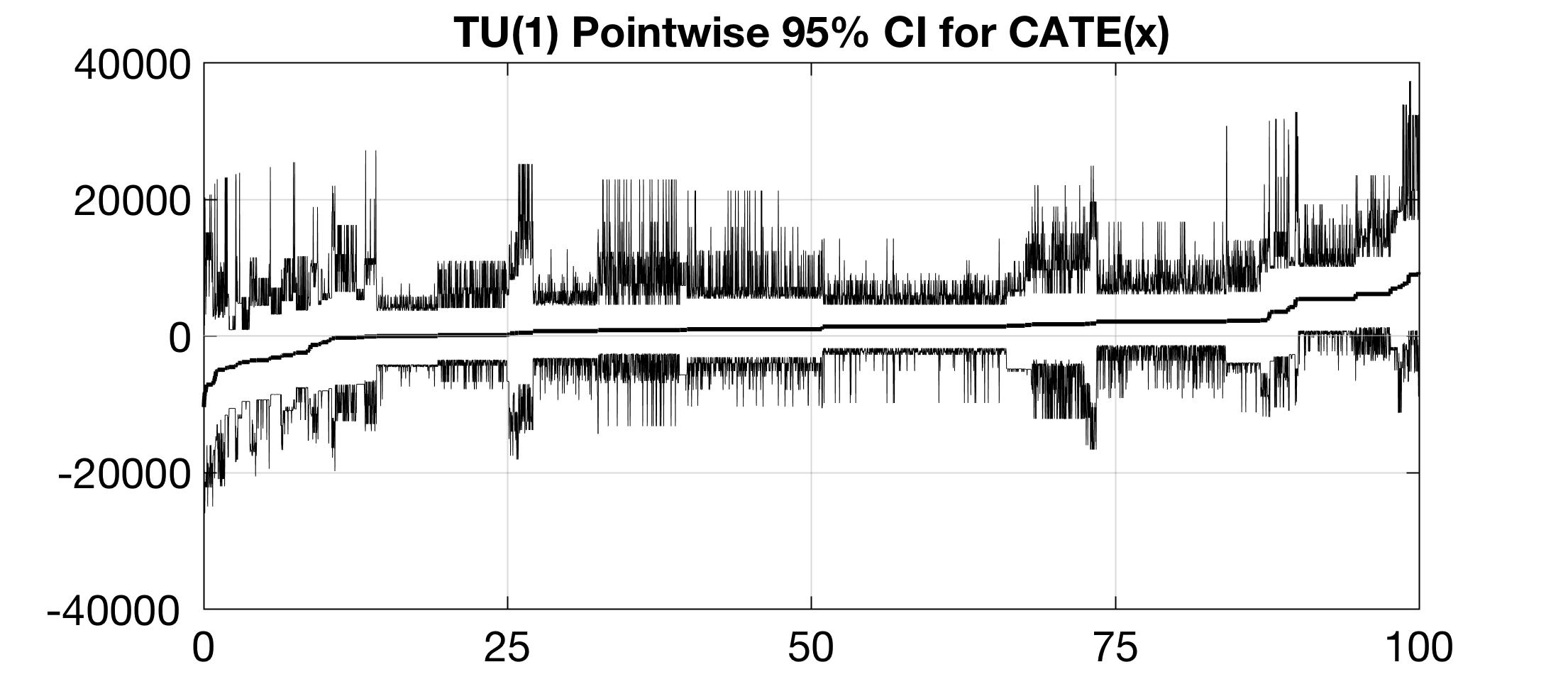}
\includegraphics[scale=.09]{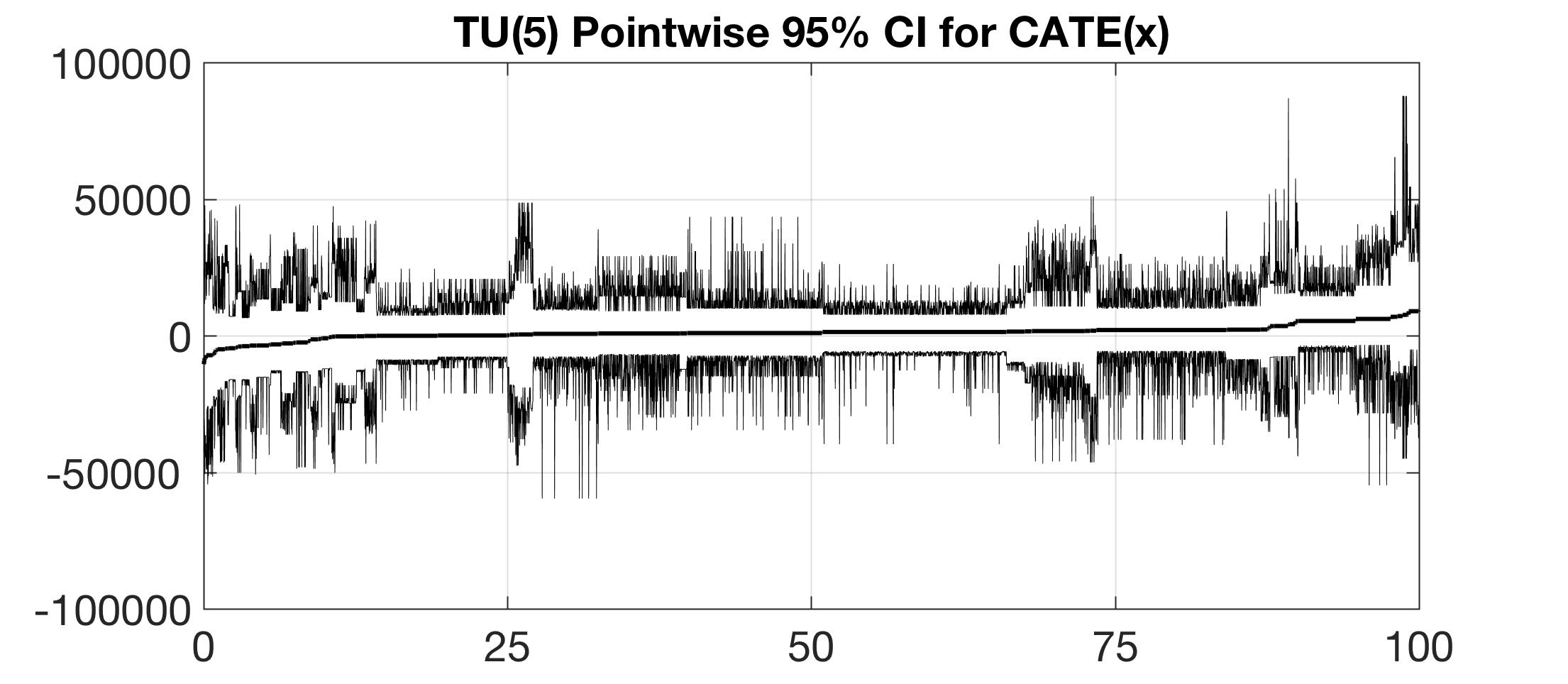}
\begin{flushleft}
\footnotesize{\textbf{Note:}  These figures report estimates of the treatment effect for each individual in the JTPA sample along with pointwise 95\% confidence intervals where the set of controls is constructed by taking all possible interactions of the baseline dummy variables and augmenting with the Hadamard-Walsh basis as described in the main text.  Estimates based on Post-lasso are reported in the top panel.  The middle and bottom panels present results based on targeted undersmoothing with $\bar{s} = 1$ (``TU(1)'') and with $\bar{s} = 5$ (``TU(5)'') respectively.  It is important to note that vertical axis is different in each figure.}
\end{flushleft}
\end{figure}

The figure shows that resulting oracle-style confidence intervals are similar to those in Figure \ref{JTPACATE}.  Both sets of interval lengths are comparatively very tight.  Though, as discussed above, the oracle-style intervals are expected to have poor performance in finite samples.   Using $\bar s=1$ we see that many of the interval lengths increase as before. There still remains a set of individuals for whom the corresponding confidence interval excludes zero.  With $\bar s = 5$, for all individuals, the corresponding intervals contain zero.  Though not pictured in Figure \ref{JTPACATEHAD}, we note that all intervals for individual-specific treatment effects include 0 as soon as $\bar s = 2$.

Finally, we again report results for testing the null hypothesis of no treatment heterogeneity, $H_0:  \gamma_0 = 0$, using the expanded model in Table \ref{JTPATestHAD}.  The procedure is implemented as before, using the standard Wald test and the results are reported in targeted undersmoothing using $\bar s \leq 10$.   We see that we reject the null hypothesis for $\bar s =  1$ at the 5\% level but fail to reject for $\bar{s} \geq 2$. 

\begin{table}[ht]
\caption{Testing the Null Hypothesis of No Treatment Effect Heterogeneity: Hadamard-Walsh Specification}
\label{JTPATestHAD}

\centering{}%
\begin{tabular}{lcccc}
\hline 
Estimator & W-statistic & df &  p-value  \tabularnewline
\hline 
PL  & 20.6884  & 9 & 0.0141  \\ 
TU(1)  & 19.4059  & 10 & 0.0354  \\ 
TU(2)  & 18.1018  & 10 & 0.0533  \\ 
TU(3)  & 17.5105  & 10 & 0.0638  \\ 
TU(4)  & 16.8746  & 10 & 0.0772  \\ 
TU(5)  & 16.3060  & 10 & 0.0912  \\ 
TU(6)  & 15.7466  & 10 & 0.1071  \\ 
TU(7)  & 15.2801  & 10 & 0.1222  \\ 
TU(8)  & 14.8188  & 10 & 0.1388  \\ 
TU(9)  & 14.3024  & 10 & 0.1596  \\ 
TU(10)  & 13.9031  & 10 & 0.1775  \\ 
\hline 
\end{tabular}
\begin{adjustwidth}{2cm}{2cm}
\footnotesize{\textbf{Note:}  This table presents results for testing the null hypothesis of no treatment effect heterogeneity when the set of controls is constructed by taking all possible interactions of the baseline dummy variables and augmenting with the Hadamard-Walsh basis as described in the main text.  We report the value of the Wald statistic (``W-statistic''), degrees of freedom (``df''), and associated p-value (``p-value'').  Results for testing this hypothesis based on OLS and Post-lasso estimates are provided in the first two rows of the table.  Rows labeled ``TU(j)'' correspond to targeted undersmoothing with $\bar{s} = j$.}
\end{adjustwidth}
\end{table}

Taken together, the results in this section suggest there is mild evidence for treatment effect heterogeneity in this example. We would reject the hypothesis of no heterogeneity and also obtain some evidence for individual specific treatment effects that differ from zero when using oracle model selection results.  However, we cannot rule out the possibility of no treatment effect heterogeneity after allowing for a modest number of model selection mistakes within either of the bases considered.  Thus, to draw strong conclusions about treatment effect heterogeneity, one must believe that the initial model selection procedure is very close to perfect in this example.

\subsection{Application II: Heterogeneous Treatment Effects in Direct Mail}\label{subsec: mail}

The targeting of individuals with appropriate interventions that induce
preferred outcomes is a relevant problem in various application areas
including business, political science and economics. In the field
of marketing, such targeting has been the key instrument of retailers
that use direct mail as the focal intervention to inform and persuade
their customers to purchase from their catalogs. These catalogs are
often relatively expensive to produce and firms spend significant
amounts in this endeavor.\footnote{In 2009, the estimated spending on catalogs was \$15.1B; and over 10B catalogs were mailed in 2015 (\cite{catalogs:nytimes}, \cite{high:costs:catalog:retailing}).}

Our data for this example comes from a large multi-product retailer that sells directly
to consumers online but also via mail, phone and retail channels.
The firm's budget for direct-mailed catalogs is over \$120M and net
sales per year are in excess of \$1.5B. The firm routinely runs experiments
to evaluate the effectiveness of its catalog mailing strategy. Typically,
these experiments have two conditions (mail, no-mail) that are randomized
across customers. Our data focuses on one such experiment that involved
over 290,000 customers. The data also include a list of 486 descriptors
of the the individual customers. These descriptors include demographic
characteristics (age, income, gender, state), details of past promotional
activity they may have received as well as their past consumption behavior data
including purchases, the timing of such purchases, the number of orders
in the past year, and the extent of their expenditures with the firm.
This last set of variables are commonly referred to as RFM (Recency,
Frequency and Monetary value) metrics in the direct mail industry
and are commonly used variables in analyzing and predicting customer
behavior. We note that the design matrix in our analysis contains 2139 columns once categorical variables are expanded.

In our analysis, we estimate the following simple specification of a model with heterogeneous treatment effects:
\begin{align*}
y_{i}=f_0(d_i,x_i,\varepsilon_i)=x_{i}\beta_{0}+(d_{i}\cdot x_{i})'\gamma_{0}+\varepsilon_{i} & .
\end{align*}
\noindent In the above, $d_{i}$ is an indicator that a consumer has been randomly assigned to receive a direct mail marketing instrument (a catalog), and the $x_{i}$ are customer characteristics. $y_{i}$ are dollar expenditures
by the customer over a 3-month horizon following the mailing of the
marketing instrument.   For notational convenience, we assume that $(x_i, \varepsilon_i)_{i=1}^n$ are $n$ i.i.d. draws, having the same distribution as the generic pair of random variables $(x,\varepsilon)$.

In this exercise, we assume that the firm is interested in evaluating a marketing strategy formed from targeting individuals based on their individual-specific treatment effects versus one of two simple baseline strategies - either mailing to no one or mailing to everyone.   To this end, we note that a mailing strategy $\tilde {d}=\tilde d(x)$ assigns customers with characteristics $x$
to either receive the mailing or not.  We then adopt targeted undersmoothing to provide a simple mechanism that allows
the firm to statistically evaluate the difference between any two
competing mailing strategies on the basis of average expected profits.
The average expected profit from implementing a strategy $\tilde {d}$
is given by
\begin{align*}
\Ep[\pi (\tilde {d})]= \Ep\left[\nu f_0\left(\tilde d(x),x,\varepsilon \right)-\tilde d(x)c \right].
\end{align*}

A few points about the above quantity are worth noting. First, the
firm has a known margin $\left(0<\nu<1\right)$ that applies to sales
generated by its customers. For simplicity, we assume that the cost to the firm of targeting each consumer, $c$, is constant and known \textit{ex ante}.\footnote{A more general approach would be to write costs as functions of $x.$ Implementing this approach would require specific data about individual mailing costs which we currently do not have.  We could also assume that costs are drawn from some known distribution where the exact realization is unknown by the firm until after the mailings have been sent out and calculate expected profits integrating over this cost distribution.}  Within the model, there is just one remaining source of uncertainty - the unanticipated demand shocks $\varepsilon$ which are only observed via outcomes - which are assumed to have conditional mean zero.

We begin by examining two extremal mailing strategies where
either no customers receive a catalog (`no-mailings') by setting $\tilde d(x)=0$
uniformly or a `blanket-mailing' strategy wherein all customers receive
a catalog (i.e. $\tilde d(x)=1$ for all $x$). For the no-mailings strategy
expected profits are 
\begin{align*}
\Ep[\pi^{0}] & =\Ep[\nu f_0 \left(0,x,\varepsilon \right)]\\
 & = \nu\Ep[x'\beta_{0}].
\end{align*}
Similarly, the expected profit for the blanket mailing strategy can
be written as 
\begin{align*}
\Ep[\pi^{1}] & =\Ep[\nu f_0(1,x,\varepsilon) - c]\\
 & = \nu\Ep[x'(\beta_{0}+\gamma_{0})] - c.
\end{align*}

A sophisticated firm might be interested in optimizing the
mailing strategy based on expected consumer response.\footnote{See \cite{AtheyWager:PolicyLearning} and \cite{luedtke:vdL:OptimalTreatment} for interesting approaches to estimating and performing inference for optimal treatment strategies.} One simple, sensible mailing strategy
would be to mail to a consumer with characteristics $x$ whenever
the expected increment in profits for that customer exceeds costs.
The rule can be described by
\begin{align*}
d^{*}(x) & =\textbf{1}\{\nu(x'\beta_{0}+x'\gamma_{0})-\nu(x'\beta_{0}) > c \}\\
 & = \textbf{1}\{\nu\left(x'\gamma_{0}\right) > c\}.
\end{align*}
Using this strategy, we then have expected per consumer profit  of 
\begin{align*}
\Ep[\pi^{*}] & =\Ep[\nu f_0(d^*(x),x,\varepsilon) - c d^*(x)]\\
 & = \nu\Ep[x\beta_{0}] + \nu\Ep[(d^*(x)\cdot x)'\gamma_{0}] - c \text{Pr}(d^*(x) = 1).
\end{align*}

Now suppose we wish to compare the targeted strategy to the `blanket' or `no-mailing' strategies. We can describe the
difference in profit between the targeted and no-mailing strategies as
\begin{align*}
\Ep[\Delta\pi^{*0}] & = \Ep\left[\pi^{*}\right]-\Ep\left[\pi^{0}\right] \\
 & = \nu\Ep[(d^*(x)\cdot x)'\gamma_{0}] - c \text{Pr}(d^*(x) = 1).
\end{align*}
Similary, the difference between the targeted and blanket strategies would be
\begin{align*}
\Ep[\Delta\pi^{*1}] & = \Ep\left[\pi^{*}\right]-\Ep\left[\pi^{1}\right] \\
 & = \nu\Ep[(d^*(x)-1)\cdot x'\gamma_{0}] - c (\text{Pr}(d^*(x) = 1)-1).
\end{align*}
We note that both of the expected per-person profit differentials capture the benefits due to cost savings and lost revenues of targeting based on expected treatment effects.  Relative to targeting no one, targeting based on anticipated treatment effect has the potential to increase revenue at the cost of paying the treatment cost for the targeted individuals.  Relative to treating everyone, targeting based on anticipated revenues has the potential to decrease costs by not targeting individuals for whom the treatment is anticipated to be ineffective.  

Simple natural estimators exist for both $\Ep[\Delta\pi^{*0}]$ and $\Ep[\Delta\pi^{*1}]$
The natural estimator for $\Ep[\Delta\pi^{*0}]$ is
\begin{align*}
\widehat{\Delta\pi^{*0}} & = \frac{\nu}{n}\sum_{i=1}^{n} \Big [\textbf{1}\{\nu\left(x_{i}'\widehat\gamma_{0}\right) > c\}\left(x_{i}'\widehat\gamma_{0} - c/\nu\right)\Big ]
\end{align*}
for some estimator $\widehat\gamma_0$.  
Similarly, a natural estimator of $\Ep[\Delta\pi^{*1}]$ is 
\begin{align*}
\widehat{\Delta\pi^{*1}} & = \frac{\nu}{n}\sum_{i=1}^{n} \Big [\left(\textbf{1}\{\nu\left(x_{i}'\widehat\gamma_{0}\right) > c\}-1\right)\left(x_{i}'\widehat\gamma_{0} - c/\nu\right)\Big ]
\end{align*}
for an estimator $\widehat\gamma_0$.  Under the sparsity assumptions on the true model maintained in this paper and conventional regularity conditions, $\widehat{\Delta\pi^{*0}}$ and $\widehat{\Delta\pi^{*1}}$ will by asymptotically normal with standard error that can be estimated via the delta-method when $\gamma_0$ is estimated from the true model.  Based on this observation, we can apply the targeted undersmoothing approach to conduct inference on potential profit improvements from targeting based on the rule $d^*(x)$ relative to the two simple baseline strategies.

We present estimates and targeted undersmoothing confidence intervals for $\Ep[\Delta\pi^{*0}] $ and $\Ep[\Delta\pi^{*1}] $ in Tables \ref{CATPi0} and \ref{CATPi1} respectively.\footnote{As with the JTPA example, before any estimation is done, variables with a very small number of nonzero observations are excluded.  In the first pass, variables with $\leq 100$ nonzero entries in the \textit{entire} sample were eliminated.  In the second pass, variables were eliminated if the corresponding diagonal R term in the design matrix QR decomposition was $<10^{-6}$ over  \textit{either} control or treated subsample.}  In all calculations, the margin parameter is set to $\nu = 0.30$ and the cost parameter is set at $c = 0.70$ based on input from the firm.  We first report OLS-based estimates, which use all covariates.  In addition, we report oracle-style post-lasso estimates as well as targeted undersmoothing estimates for $\bar s \leq 10$.  We implement the first stage model selection using the procedure in Appendix 1.  We use heteroskedasticity consistent standard errors and calculate confidence intervals using the delta method.

\begin{table}[htbp]
\caption{Estimates for Average Profit Differential Relative to No Mailing: $$\Ep[\Delta\pi^{*0}] $$ }
\label{CATPi0}

\centering{}
\begin{tabular}{lcccc}
\hline 
Estimator & Estimate & S.E. & Lower & Upper\tabularnewline
\hline 
OLS	&	1.1514	&	0.0655	&	1.0229	&	1.2798	\\
PL	&	0.6984	&	0.0441	&	0.6119	&	0.7849	\\
TU(1)	&		&		&	0.6099	&	0.7960	\\
TU(2)	&		&		&	0.6083	&	0.8063	\\
TU(3)	&		&		&	0.6070	&	0.8131	\\
TU(4)	&		&		&	0.6062	&	0.8188	\\
TU(5)	&		&		&	0.6054	&	0.8269	\\
TU(6)	&		&		&	0.6045	&	0.8323	\\
TU(7)	&		&		&	0.6036	&	0.8375	\\
TU(8)	&		&		&	0.6029	&	0.8430	\\
TU(9)	&		&		&	0.6023	&	0.8476	\\
TU(10)	&		&		&	0.6018	&	0.8514	\\
\hline 
\end{tabular}
\begin{adjustwidth}{2cm}{2cm}
\footnotesize{\textbf{Note:}  This table presents estimates of the average profit differential between the targeted mailing strategy and the strategy that mails to no one.  OLS and Post-lasso estimates of the average profit differential and associated standard errors are provided in the ``Estimate'' and ``S.E.'' columns in the first two rows.  The ``Lower'' and ``Upper'' columns respectively report the lower and upper bounds of 95\% confidence intervals.  Rows labeled ``TU(j)'' correspond to targeted undersmoothing with $\bar{s} = j$.}
\end{adjustwidth}
\end{table}

\begin{table}[htbp]
\caption{Estimates for Average Profit Differential Relative to Uniform Mailing:
$$\Ep[\Delta\pi^{*1}] $$}
\label{CATPi1}

\centering{}%
\begin{tabular}{lcccc}
\hline 
Estimator & Estimate & S.E. & Lower & Upper\tabularnewline
\hline 
OLS	&	0.6332	&	0.0789	&	0.4785	&	0.7879	\\
PL	&	0.1811	&	0.0497	&	0.0837	&	0.2784	\\
TU(1)	&		&		&	0.0821	&	0.2905	\\
TU(2)	&		&		&	0.0807	&	0.3001	\\
TU(3)	&		&		&	0.0798	&	0.3076	\\
TU(4)	&		&		&	0.0788	&	0.3132	\\
TU(5)	&		&		&	0.0779	&	0.3205	\\
TU(6)	&		&		&	0.0773	&	0.3261	\\
TU(7)	&		&		&	0.0767	&	0.3309	\\
TU(8)	&		&		&	0.0762	&	0.3361	\\
TU(9)	&		&		&	0.0758	&	0.3401	\\
TU(10)	&		&		&	0.0754	&	0.3437	\\
\hline 
\end{tabular}
\begin{adjustwidth}{2cm}{2cm}
\footnotesize{\textbf{Note:}  This table presents estimates of the average profit differential between the targeted mailing strategy and the strategy that mails to everyone.  OLS and Post-lasso estimates of the average profit differential and associated standard errors are provided in the ``Estimate'' and ``S.E.'' columns in the first two rows.  The ``Lower'' and ``Upper'' columns respectively report the lower and upper bounds of 95\% confidence intervals.  Rows labeled ``TU(j)'' correspond to targeted undersmoothing with $\bar{s} = j$.}
\end{adjustwidth}
\end{table}

We see that the confidence intervals for the parameters $\Ep[\Delta\pi^{*0}] $ and $\Ep[\Delta\pi^{*1}] $ are very robust to different assumptions about the true underlying sparsity level $\bar s$.  Interestingly, the OLS-based intervals are completely different from the targeted undersmoothing intervals for every value of $\bar s$ reported.  This difference is likely due to a failure of OLS in this example.  In the setting of the simulation study below, we find that OLS-based intervals achieve poor coverage probabilities with coverages as low as 0.00\% in some settings.  The poor performance of OLS in the simulation study is due to biases arising from taking a nonlinear transformation of the estimated coefficient vector and a failure of the standard delta method with a large number of covariates.\footnote{Bias corrections for the delta method in settings with many covariates are described in \cite{cattaneo2017}.  For simplicity, we report the estimates and intervals which correspond to common practice.} In this example, the OLS-based estimates seem to overstate both $\Ep[\Delta\pi^{*0}] $ and $\Ep[\Delta\pi^{*1}]$.

Finally, we test the hypothesis $H_0:  \gamma_0 = 0$ in Table \ref{CATTest}.  As in the previous example, this hypothesis corresponds to the hypothesis of no treatment effect heterogeneity.  From a policy standpoint, understanding whether there is evidence for treatment effect heterogeneity may be interesting as there is clearly no gain from any targeting strategy based on observables if the treatment effect is constant across these observables.  The results for testing this hypothesis are presented in Table 5. We note that the OLS-based result is likely unreliable due to relying on a heteroskedasticity-consistent estimate of a large, full covariance matrix, but we report the result for completeness. In this example, we see that the p-values are very near zero for all considered values of $\bar s$, suggesting that there is strong evidence against the hypothesis of no treatment effect heterogeneity that is robust to fairly large deviations from the initially selected model. As in the previous example, we also see that the degrees of freedom of the test is constant across the different values of $\bar{s}$ indicating that the additional variables being added all enter the model via the $\x_i\beta_0$ term.  Adding variables to this part of the model that are correlated to the estimated treatment effect reduces the signal available to learn about treatment effect heterogeneity and thus intuitively provides ``worst-case'' deviations from the standpoint of drawing conclusions about the existence of this heterogeneity.

\begin{table}[htbp]
\caption{Testing the Null Hypothesis of No Treatment Effect Heterogeneity}
\label{CATTest}

\centering{}%
\begin{tabular}{lcccc}
\hline 
Estimator & W-statistic & df & p-value  \tabularnewline
\hline 
OLS	&	1865.7525 & 1069 & 0.000		\\
PL	&	692.4930 & 45&0.000		\\
TU(1)	&	  685.5655 & 45& 0.000		\\
TU(2)	&    680.9011 & 45&	0.000			\\
TU(3)	&  678.0659 & 45&	0.000	\\
TU(4)	&	  675.3192 & 45&	0.000	\\
TU(5)	&  672.9171 & 45& 0.000	\\
TU(6)	&	  671.3020 & 45& 0.000		\\
TU(7)	&	  669.6907 & 45&	0.000\\
TU(8)	&	 668.4609 & 45& 0.000		\\
TU(9)	&	  667.4802 & 45& 0.000		\\
TU(10)	&	  666.4816 & 45& 0.000		\\
\hline 
\end{tabular}
\begin{adjustwidth}{2cm}{2cm}
\footnotesize{\textbf{Note:}  This table presents results for testing the null hypothesis of no treatment effect heterogeneity.  We report the value of the Wald statistic (``W-statistic''), degrees of freedom (``df''), and associated p-value (``p-value'').  Results for testing this hypothesis based on OLS and Post-lasso estimates are provided in the first two rows of the table.  Rows labeled ``TU(j)'' correspond to targeted undersmoothing with $\bar{s} = j$.}
\end{adjustwidth}
\end{table}

\section{Simulation Study}

In this section, we present a simulation study designed to demonstrate the properties of the proposed procedure in finite samples.  We consider six simulation designs based on the example in Section \ref{subsec: mail}.  We generate data for each simulation replication as iid draws for $i = 1,...,n$ from the model
\begin{align*}\label{sim: true}
&y_i = \alpha_0 + x_i'\beta_0 + d_i \gamma_0 + d_i\cdot x_i' \zeta_0 +  \varepsilon_i,
\end{align*}
\begin{align*}
&p = 2+ 2 \text{dim}(x_i) = 2(1+k),\\
&w_{ij} \sim N(0,1) \ \text{with corr}(w_{ij_1},w_{ij_2}) = 0.8^{|j_1-j_2|}, \\
&x_{ij} = (w_{ij} - \tau_j)\textbf{1}\{ w_{ij} \geq  \tau_j\},\\
&\tau_j \sim \text{unif}(0,1.28), \ {iid}, \\
& d_i \sim \text{Bernoulli}(0.5),\\
&\varepsilon_i \sim N(0,1), \\
& (\alpha_0,\beta_0') = c_{.25} (1/\sqrt{s_0},(2/\sqrt{s_0})\iota_{s_0/4}',(2/\sqrt{ns_0})\iota_{s_0/4}',0_{k-s_0/2}') \odot (1,\upsilon'), \\
& (\gamma_0,\zeta_0') = c_{.25} (1/(2\sqrt{s_0}),(4/\sqrt{ns_0})\iota_{s_0/4}',(4/\sqrt{s_0})\iota_{s/4}',0_{k-s_0/2}') \odot (1,\upsilon'), 
\end{align*}
where $c_{.25}$ is a constant that is chosen so that the population $R^2$ of the regression of $y_i$ onto $(1,x_i',d_i,d_i x_i')$ is $0.25$, $\iota_m$ is an $m \times 1$ vector of ones, $0_m$ is an $m \times 1$ vector of zeros, $\upsilon$ is a $k \times 1$ vector with $j^{th}$ element given by $\upsilon_j = (-1)^{j-1}$, and $\odot$ denotes the Hadamard product.  The six considered simulation designs are based on varying $p \in \{202,602\}$ and $s_0 \in \{4,8,16\}$.  In all simulations, we take $n = 400$.  We note that the process for the $x_{ij}$ is meant to approximate what we see in the observables in the example in Section \ref{subsec: mail} which are all positive with large fractions of observations exactly at 0.  For each simulation design, we estimate and construct confidence sets for three functionals: (1) the value of a single coefficient (specifically $\zeta_{0,1}$), (2) an individual treatment effect for a fixed hypothetical subject (with $x^* = .5\iota_{\text{dim}(x_i)}$), and (3) the average per-person profit differential from a targeting rule based on estimated individual specific treatment effects and a rule which treats no one ($\Ep[\Delta\pi^{*0}]$ defined in Section \ref{subsec: mail}).

For each set of model parameters, we simulate 500 replications and present the properties of several estimators:

\

\begin{itemize}
\item[1.]  \textbf{True.} An infeasible estimator based on ordinary least squares on the correct support of the underlying model.  
\item[2.]  \textbf{All.} An estimator based on ordinary least squares using all covariates.  
\item[3.]  \textbf{Double.}  The post-double estimator as described in \cite{BCH-PLM}
\item[4.]  \textbf{Lasso.}  An estimator based on lasso.  Standard errors computed using lasso residuals.
\item[5.] \textbf {PL.} An estimator based on the post-lasso estimator of \cite{BC-PostLASSO}.  Standard errors computed using post-lasso residuals.
\item[6.]  \textbf{LCV.} An estimator based on lasso with penalty level chosen by 10-fold cross validation.  Standard errors are computed using lasso residuals.
\item[7.]  \textbf{ZB.}  Confidence intervals based on inverting the hypothesis test prosed in \cite{ZhuBradic:linear}.
\item[8.]  \textbf{TU(1).}  Targeted undersmoothing with $\overline s = 1$ using Algorithms 1 and 3.  Initial model $\hat S^0$ description in Implementation Appendix.  
\item[9.]  \textbf{TU(10).} Targeted undersmoothing with $\overline s = 10$ using Algorithms 1 and 3. Initial model $\hat S^0$ description in Implementation Appendix.
\end{itemize} 

\

All standard errors are computed using conventional heteroskedasticity consistent standard errors (e.g. \cite{white:het}) using the estimated residuals indicated above.  We give details on implementation specifics in the following paragraphs.\footnote{There are many choices about how to implement the different procedures, e.g. whether to split into treatment and control observations and which penalty parameters to use.  The choices below were based on initial simulations where they seemed to produce the most favorable performance for the non-targeted undersmoothing approaches.}

For True, All, and Double, we directly estimate the model above.  For Double, we apply \cite{BCH-PLM} with a minor modification.  We implement the relevant lasso regressions from \cite{BCH-PLM} using the modified heteroskedastic lasso of Appendix 1.  

To implement lasso, PL, we use the implementation given in Appendix 1 to select a model.  The PL estimates re-estimate coefficients by applying OLS with only the variables selected by lasso.  For LCV, we use a modification of the procedure in Appendix 1, where 10-fold cross-validation within each subset is used to choose the tuning parameter to use in that subset.  We then apply the conventional lasso within each subset based on these estimated tuning parameters.  For these methods, we then can obtain estimates and standard errors for the functionals of interest in the obvious manner.  ZB implements the proposed method of inference for dense linear functionals of a parameter vector from \cite{ZhuBradic:linear}.
Finally, the PL model serves as our initial model when applying targeted undersmoothing.  We apply targeted undersmoothing for $\bar{s} = 1,...,10$.

To measure the performance of the nine procedures, we report estimates of bias, standard deviation, root mean-square error (RMSE), coverage probability for a 95\% confidence interval, and corresponding confidence interval length from the simulation in {Tables Sim1-Sim6 and Figures Sim1-Sim6}.  In the figures, we provide average confidence interval lengths and coverage probabilities along the 10-steps of the forward selection path produced in the simulation.  As a benchmark, we superimpose coverage probabilities and interval lengths for the infeasible `True' estimator which knows the correct model on the targeted undersmoothing path plots.   

The `True' estimator provides an infeasible benchmark which serves as a basis for comparison.  In most simulations, the `True'  estimator achieves the target 95\% coverage probability.  In general, the `True' estimator also achieves the smallest bias, RMSE, and shortest confidence intervals.  All other estimators provide feasible alternatives that ideally would approximate the behavior of this infeasible benchmark.

When the number of parameters to be estimated is smaller than the sample size, a simple feasible option is to estimate the full-model without any model selection.  In terms of our simulation, this approach clearly results in small bias for the individual regression parameter and for the individual-specific treatment effect as both of these objects are linear combinations of the regression coefficients and the variables in the design are mean-independent of the error term.  The cost of estimating the full model is decreased estimation precision as evidenced by relatively large standard deviation and RMSE relative to the other point estimators.  We also see that the confidence intervals produced after estimating the full model are relatively long, often longer than the intervals resulted from targeted undersmoothing with small or moderate $\bar{s}$.  The most interesting feature of the results based on the full model are for estimating the profit differential.  For this object, the estimator is dominated by bias due to the profit differential depending nonlinearly on the model parameters and the imprecision in estimating these parameters.  This bias then results in very poor coverage properties for the true profit differential.  This behavior can be viewed as a failure of the delta-method in moderate or high-dimensional models; see \cite{cjm:nonlinear}.  We suspect this behavior will carry over to many nonlinear settings.

We next examine the performance of `Lasso' and `PL'.  We note that the lasso penalty parameter in this case is set in a manner that theoretically provides lasso with an optimal rate of convergence and guarantees that the $\hat{s} = O(1)s_0$.  We then conduct inference in these cases by relying on oracle-type results (see for example \cite{Zou2006}, \cite{BuneaTsybakovWegkamp2007b}) that ignore the first step model selection.  These estimators behave roughly as expected by theory.  In general, the estimators are competitive in terms of RMSE for all objects considered across all different designs.  However, their bias also tends to be comparable to their standard deviation due to regularization and model selection mistakes.  Oracle-style approximations do not account explicitly for this remaining bias due to regularization and as a result do not achieve correct coverage rates.  We note that these distortions can be severe.  Coverage for these procedures is generally far from the nominal 95\%; and in some cases, the estimators have 0\% coverage.  We note that targeted undersmoothing is expressly designed to offer a generic approach to address the presence of this bias.

The `LCV' estimator is similar to `Lasso' and `PL' in that it applies oracle-style inference after selecting a model from the data.  The difference is that cross-validation tends to produce penalty parameters that are much smaller than the theoretically motivated values used in `Lasso' and `PL'.  This reduction in the penalty parameter allows extra variables to enter the model relative to the case where the larger penalty parameters are used.  In this sense, such a procedure can also be thought of as an undersmoothing procedure, though the ``undersmoothing'' is targeted toward model fit.\footnote{\cite{ccl:cv} demonstrates that cross-validation may produce estimates with slower than optimal convergence rates with models that are much too complex in the sense that $\hat{s} \gg s_0$.}  In these simulations, we see that LCV tends to produce estimates of the regression coefficient and individual-specific treatment effect with bias similar to that obtained with lasso and PL, though LCV also tends to have a larger standard deviation than these estimators as well.  The similar bias and larger standard deviation results in LCV tending to be outperformed in terms of RMSE for these objects but also results in better coverage properties of the LCV intervals than the lasso or PL intervals - though LCV coverage still tends to be far from the nominal level.\footnote{Exceptions are coverage of the individual specific treatment effect in Tables Sim1, Sim2, Sim4, Sim5.}   For the profit differential, LCV is less-biased than lasso in all cases and less-biased than PL in four of six cases while generally having similar standard deviation.  Thus, LCV is competitive in terms of RMSE for this object.  However, sufficient bias remains for confidence intervals to remain substantively distorted, producing coverage probabilities for the profit differential that range between 0.63 and 0.90.

In many studies, the object of interest is an inherently low-dimensional parameter, such as a single regression coefficient or an average treatment effect, and semi-parametric estimation can be designed that specifically targets this low-dimensional parameter of interest.\footnote{See, for example, \cite{bickel1998book}, \cite{vdV-W}, \cite{newey:semiparametric:1994}, \cite{vanderLann:Daniel:2006targeted} for classic examples.  \cite{DML} provide a recent treatment in a high-dimensional setting.}  This approach is adopted in the high-dimensional linear model setting in \cite{BCH-PLM}, \cite{vdGBRD:AsymptoticConfidenceSets} and \cite{ZhangZhang:CI} for estimating a single regression coefficient of interest.  For regression coefficients, these procedures are $\sqrt{n}$-consistent and semi-parametrically efficient within the model considered in the simulation.  They also theoretically deliver uniformly valid inference over large classes of models which include cases where perfect model selection is theoretically impossible.  In terms of our simulations, this approach does relatively well in the $s_0 = 4$ case, delivering performance which is comparable to the infeasible oracle.  However, in the $s_0 = 8$ and $s_0 = 16$ cases, the point estimator has a large bias which translates into relatively poor coverage properties.\footnote{The behavior may be improved by considering double machine learning as defined in \cite{DML}, which relies on weaker sparsity conditions than \cite{BCH-PLM}. We note that targeted undersmoothing offers an approach to gauging the sensitivity of conclusions to model selection mistakes and could be applied directly to semiparametric targets using orthogonal estimating equations as in \cite{BCH-PLM} or \cite{DML}.  We do not pursue this direction further in this paper for brevity.}

The `ZB' method does not achieve 95\% coverage for the regression coefficient $\zeta_{0,1}$ in any of the simulation designs considered here (with coverages ranging from 73\% to 83\%).  The ZB method gives better coverage probabilities for the individual treatment effect with near or above 95\% coverage in all simulation designs.  The lengths of the ZB confidence intervals grow considerably with the underlying value of $s_0$.  For instance, in the $p=202, s_0=4$ case, the mean ZB interval length is 3.41 while the `True' mean interval length of 0.88;  in the $p=202, s_0=16$ case, the mean ZB interval length is 56.53 while the `True' mean interval length of 1.41.  

We now look at intervals constructed using the targeted undersmoothing approach.  Note that we take the initial model to be that underlying PL in these simulations, and, for point estimation, one could use these PL point estimates.  The point of targeted undersmoothing is to provide valid inferential statements allowing for model selection mistakes in producing this initial model and corresponding point estimates.  An interesting feature of the presented simulations is that TU(1) achieves nearly correct coverage uniformly across the simulation designs - achieving higher than 90\% coverage in every design.  While not reported in the table, we also have that TU(2) achieves higher than 95\% coverages in all cases.  We do see the inherent conservativeness in sensitivity analysis considering a large class of models in that TU(10) uniformly has coverage greater than 95\%, with coverage of 100\% in most cases.  Importantly, the good coverage properties are uniform across all designs and all parameters considered.  Unsurprisingly, this robustness comes with a cost.  As must be the case, the intervals produced by the targeted undersmoothing approach are relatively wide and become wider as one allows for more selection mistakes.  However, the losses relative to the infeasible optimum are modest for small $\bar{s}$ and that the intervals are still potentially informative even in the most extreme case we consider.  

Overall, we believe these results are favorable to the targeted undersmoothing approach.  Of the considered feasible alternatives, it is the only procedure that produces uniformly good coverage properties, at the cost of increased imprecision about what conclusions can be drawn from the data.  This increase in imprecision seems honest as it reflects the potential for substantive biases resulting from model selection mistakes.  The procedure is also anchored on initial point estimates that have relatively good properties for estimating the parameters of interest.   

\section{Conclusion}

In this paper, we have considered post model selection inference for a large class of functionals of the underlying model.  Our procedure provides valid confidence sets while handling the possibility that a misspecified model was selected. We show that these methods perform well in a simulation study.  We illustrate their use in estimating the profit differential for a fixed coupon-mailing strategy and in estimating heterogeneous treatment effects in data from a job training experiment.

\section*{Appendix 1.  Implementation Details}

This appendix describes the model selection procedure implemented in several sections of the paper.  Recall that the general model estimated is given by

$$y_i =  x_i' \beta_0 +( d_i \cdot x_i)' \gamma_0 + \varepsilon_i.$$

The procedure for selecting $\hat S^0$ is as follows.

\

\noindent \textit{Algorithm A1.} Initial model selection in heterogeneous effects linear model.

\noindent \textbf{Step 1.} Divide the sample into two sets:  $A_0 = \{ i : d_i = 0\}$ and $A_1 = \{ i: d_i = 1 \}$.

\noindent \textbf{Step 2.} Within each sample, demean the observations.

\noindent \textbf{Step 3.} Using the demeaned observations, run the modified heteroskedastic lasso regression (described below in Algorithm 2) of $y_i$ on $x_i$ over subset $A_0$ and let $\hat S^{0,0}$ be the set of covariates selected. Again using the demeaned observations, run the modified heteroskedastic lasso regression of $y_i$ on $x_i$ over subset $A_1$ and let $\hat S^{0,1}$ be the set of covariates selected. 

\noindent \textbf{Step 4.} 
The final model $\hat S^0$ consists of the constant term, the main effect of $d_i$, the $\beta_0$ components corresponding to covariate indexes in $\hat S^{0,0} \cup \hat S^{0,1}$, and the interaction terms ($\gamma_0$ terms) corresponding to covariate indexes in $ \hat S^{0,0} \cup \hat S^{0,1}$.

\

\noindent \textit{Algorithm A2.} Modified Heteroskedastic Lasso:  Marginal Correlation-Based Initial Penalty Loadings.
 The modified heteroskedastic lasso is identical to \cite{BellChenChernHans:nonGauss}  with a small modification.  \cite{BellChenChernHans:nonGauss}  relies on `initial penalty loadings,' which require initial estimates of individual specific residuals.  To obtain initial estimates of residuals, $e_i^{initial}$, we regress $y_i$ on the 5 covariates with the highest marginal correlation with $y_i$ and use the resulting residuals.  This approach can be shown to be formally valid when the number of covariates with high marginal correlations to $y_i$ used is bounded by a constant which does not depend on $n$.  In contrast, note that \cite{BellChenChernHans:nonGauss} suggest $e_i^{initial} = y_i - \bar y$.  Finally, the penalty loadings are updated with one iteration as described in \cite{BellChenChernHans:nonGauss}.

\bibliographystyle{plain}
\bibliography{dkbib1}

\pagebreak 
\begin{center} 
\begin{table}[H]   

{\small \textsc{Table Sim1.} Simulation Results: $n=400$, $p = 202$, $s_0 = 4$ } 

\vspace{.05in} 

\begin{tabular*}{\textwidth}{p{1.7cm} p{.9cm}  p{.8cm} p{.8cm} p{.8cm} p{.8cm} p{.8cm}  p{.8cm}  p{.8cm} p{.8cm}}

\hline          \hline                                                                  \\ 
&	\ True \		& All		&	Double	&	Lasso	&	PL & LCV	&	ZB& TU(1)	&	TU(10) 	\\
\cline{2-10}     & \multicolumn{9}{c}{A. RegCoef }\\ \cline{2-10} 
Bias	&	0.04		&	0.05	&	0.09	&	-0.13	&	-0.19	&	-0.33	&	 	&	 	\\
Std. Dev.	&	0.68		&	0.79	&	0.62	&	0.11	&	0.37	&	0.58	&	 	&	 	\\
RMSE	&	0.68		&	0.79	&	0.63	&	0.16	&	0.41	&	0.66	&	 	&	 	\\
Coverage	&	0.91		&	0.91	&	0.93	&	0.14	&	0.10	&	0.54	&	0.76 & 0.93	&	0.97	\\
Int. Length	&	2.46		&	2.70	&	2.26	&	0.28	&	0.33	&	1.42	&	0.98 & 1.97	&	3.86	\\
\cline{2-10}     & \multicolumn{9}{c}{B. TE }\\ \cline{2-10} 
Bias	&	0.01		&	-0.00	&	 	&	0.27	&	-0.01	&	-0.00	&	 	&	 	\\
Std. Dev.	&	0.24		&	1.57	&	 	&	0.15	&	0.30	&	0.35	&	 	&	 	\\
RMSE	&	0.25		&	1.57	&	 	&	0.31	&	0.30	&	0.35	&	 	&	 	\\
Coverage	&	0.91		&	0.94	&	 	&	0.56	&	0.76	&	0.94	& 0.95&	0.98	&	1.00	\\
Int. Length	&	0.88		&	5.74	&	 	&	0.67	&	0.65	&	1.49	&3.41&	1.76	&	5.44	\\
\cline{2-10}     & \multicolumn{9}{c}{C. PI }\\ \cline{2-10} 
Bias	&	0.01		&	0.32	&	 	&	-0.14	&	-0.01	&	-0.05	&	 	&	 	\\
Std. Dev.	&	0.06		&	0.07	&	 	&	0.01	&	0.08	&	0.06	&	 	&	 	\\
RMSE	&	0.06		&	0.33	&	 	&	0.14	&	0.08	&	0.08	&	 	&	 	\\
Coverage	&	0.95		&	0.00	&	 	&	0.06	&	0.81	&	0.82	& &	0.94	&	1.00	\\
Int. Length	&	0.26		&	0.27	&	 	&	0.02	&	0.22	&	0.22	&	& 0.30	&	0.45	\\
\hline
\end{tabular*}													
\vspace{.1in}													
\end{table}		
\end{center}	
\center {\small \textsc{Fig Sim1.} Simulation Results: $n=400$, $p = 202$, $s_0 = 4$ } 
\smallskip
\hrule
\textcolor{white}{\footnotesize ,}
\hrule
\center \hspace{-1.5cm}\includegraphics[scale=.30]{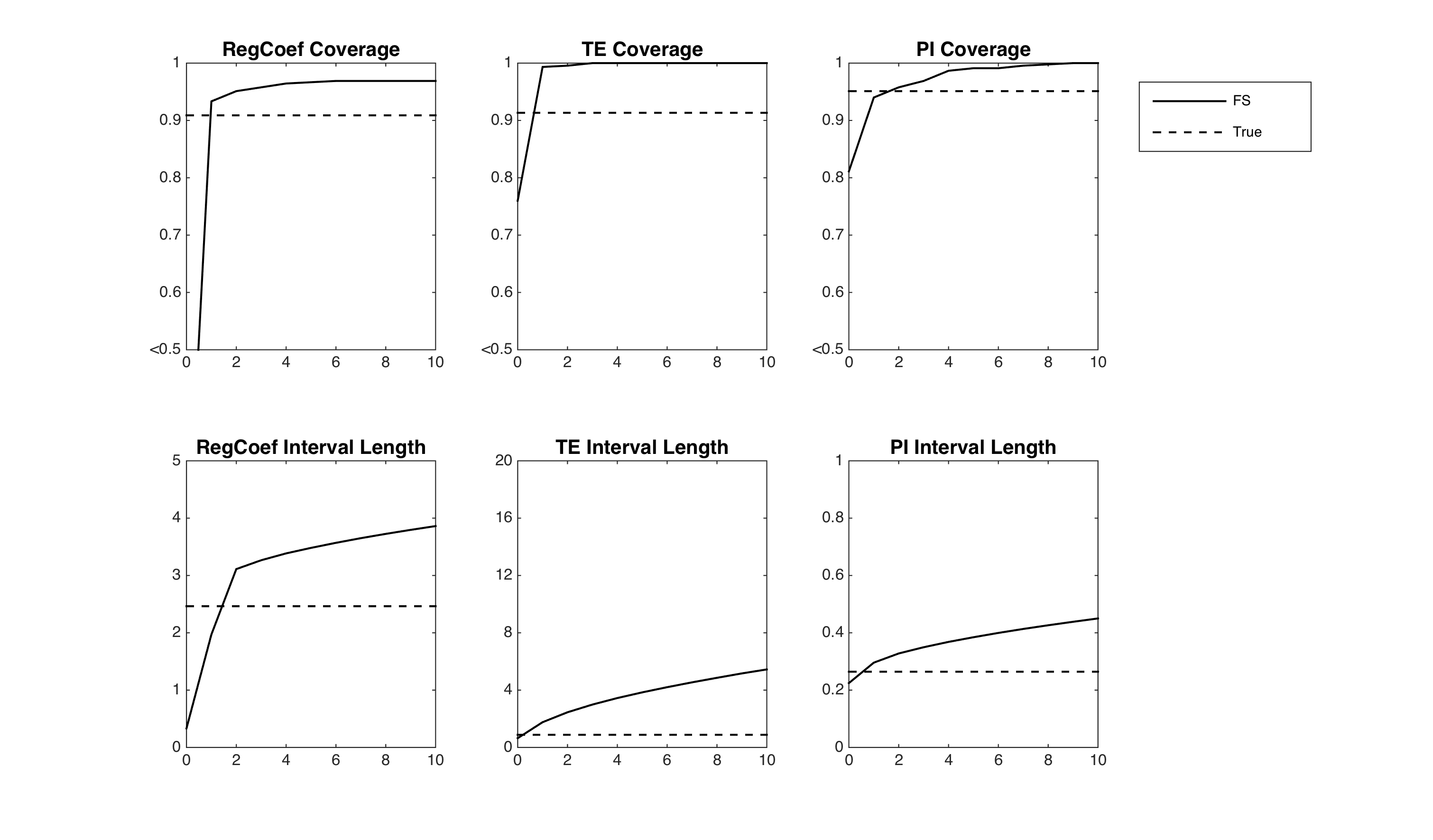}

\

\hrule

\begin{center} 
\begin{table}[H]   

{\small \textsc{Table Sim2.} Simulation Results: $n=400$, $p = 202$, $s_0 = 8$ } 

\vspace{.05in} 

\begin{tabular*}{\textwidth}{p{1.7cm} p{.9cm}  p{.8cm} p{.8cm} p{.8cm} p{.8cm} p{.8cm}  p{.8cm}   p{.8cm}  p{.8cm} }

\hline          \hline                                                                  \\ 
&	\ True \		& All		&	Double	&	Lasso	&	PL & LCV	& ZB&	TU(1)	&	TU(10) 	\\
\cline{2-10}     & \multicolumn{9}{c}{A. RegCoef }\\ \cline{2-10} 
Bias	&	0.04		&	0.01	&	0.84	&	-0.09	&	-0.08	&	-0.14	&	 	&	 	\\
Std. Dev.	&	0.63		&	0.74	&	0.67	&	0.02	&	0.18	&	0.55	&	 	&	 	\\
RMSE	&	0.63		&	0.74	&	1.07	&	0.09	&	0.19	&	0.57	&	 	&	 	\\
Coverage	&	0.94		&	0.92	&	0.67	&	0.02	&	0.01	&	0.64	& 0.79&	0.99	&	1.00	\\
Int. Length	&	2.25		&	2.61	&	2.33	&	0.04	&	0.03	&	1.51	&1.22&	2.12	&	4.25	\\
\cline{2-10}     & \multicolumn{9}{c}{B. TE }\\ \cline{2-10} 
Bias	&	0.02		&	0.01	&	 	&	0.12	&	0.13	&	0.13	&	 	&	 	\\
Std. Dev.	&	0.21		&	1.57	&	 	&	0.12	&	0.27	&	0.45	&	 	&	 	\\
RMSE	&	0.21		&	1.57	&	 	&	0.17	&	0.30	&	0.47	&	 	&	 	\\
Coverage	&	0.94		&	0.92	&	 	&	0.87	&	0.76	&	0.97	& 0.91&	0.99	&	1.00	\\
Int. Length	&	0.78		&	5.79	&	 	&	0.56	&	0.58	&	1.88	& 19.57&	2.16	&	6.72	\\
\cline{2-10}     & \multicolumn{9}{c}{C. PI }\\ \cline{2-10} 
Bias	&	0.02		&	0.31	&	 	&	-0.09	&	-0.07	&	-0.02	&	 	&	 	\\
Std. Dev.	&	0.10		&	0.10	&	 	&	0.11	&	0.11	&	0.11	&	 	&	 	\\
RMSE	&	0.10		&	0.33	&	 	&	0.14	&	0.13	&	0.11	&	 	&	 	\\
Coverage	&	0.95		&	0.06	&	 	&	0.86	&	0.87	&	0.90	&	& 0.95	&	1.00	\\
Int. Length	&	0.40		&	0.36	&	 	&	0.44	&	0.43	&	0.39	&	& 0.50	&	0.74	\\
\hline
\end{tabular*}													
\vspace{.1in}													
\end{table}		
\end{center}	
\center {\small \textsc{Fig Sim2.} Simulation Results: $n=400$, $p = 202$, $s_0 = 8$ }
\smallskip 
\hrule
\textcolor{white}{\footnotesize ,}
\hrule 

\center \hspace{-1.5cm}\includegraphics[scale=.30]{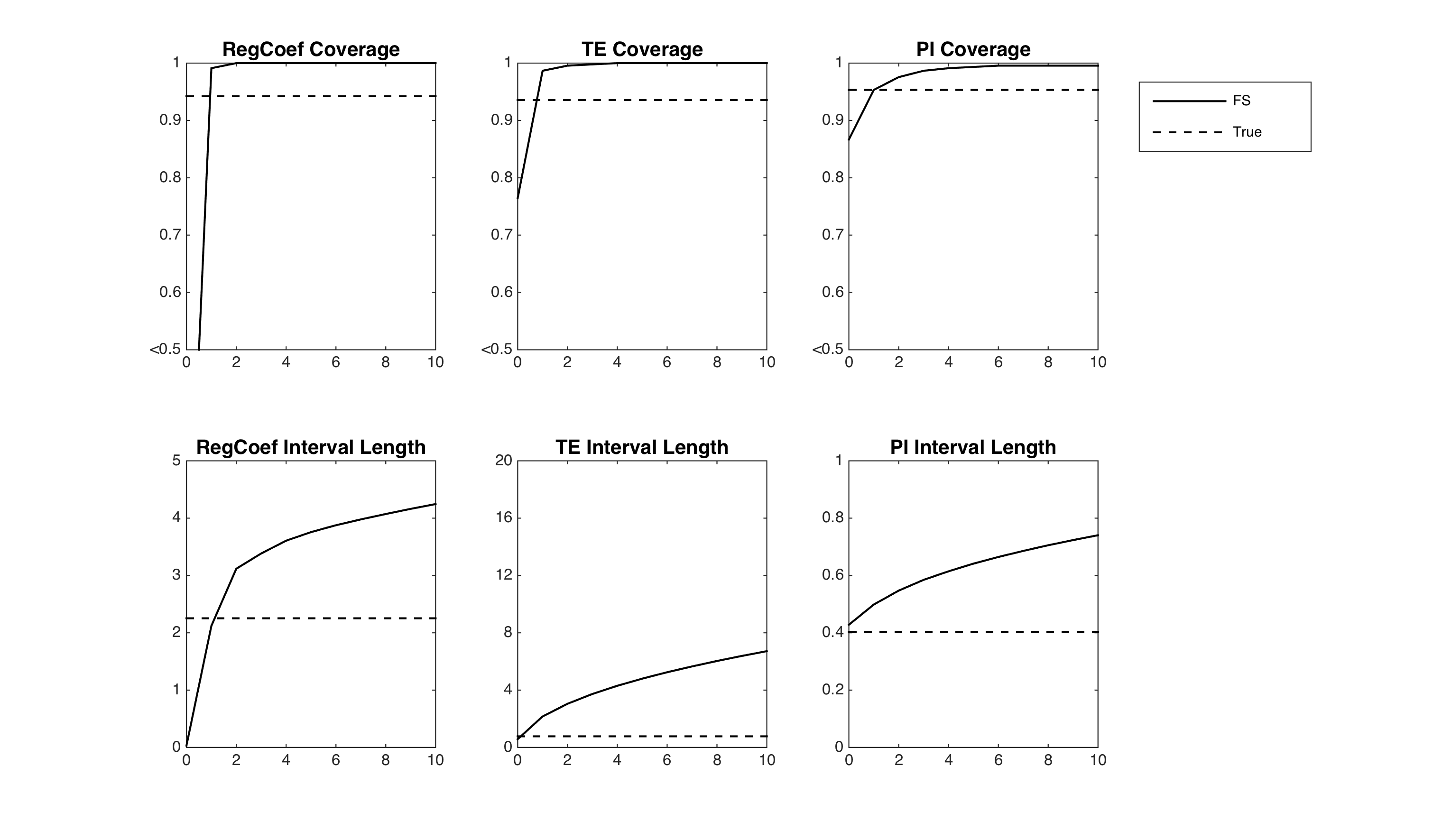}

\

\hrule

\begin{center} 
\begin{table}[H]   

\center {\small \textsc{Table Sim3.} Simulation Results: $n=400$, $p = 202$, $s_0 = 16$ } 

\vspace{.05in} 

\begin{tabular*}{\textwidth}{p{1.7cm} p{.9cm}  p{.8cm} p{.8cm} p{.8cm} p{.8cm} p{.8cm}  p{.8cm}   p{.8cm}  p{.8cm} }

\hline          \hline                                                                  \\ 
&	\ True \		& All		&	Double	&	Lasso	&	PL & LCV	&	ZB &  TU(1)	&	TU(10) 	\\
\cline{2-10}     & \multicolumn{9}{c}{A. RegCoef }\\ \cline{2-10} 
Bias	&	0.05		&	0.04	&	0.46	&	-0.07	&	-0.07	&	-0.12	&	 	&	 	\\
Std. Dev.	&	0.57		&	0.71	&	0.56	&	0.00	&	0.04	&	0.37	&	 	&	 	\\
RMSE	&	0.58		&	0.71	&	0.73	&	0.07	&	0.08	&	0.39	&	 	&	 	\\
Coverage	&	0.92		&	0.92	&	0.82	&	0.00	&	0.00	&	0.47	&	0.83 & 0.99	&	1.00	\\
Int. Length	&	2.04		&	2.44	&	2.05	&	0.00	&	0.01	&	0.85	& 1.50 &	1.36	&	3.86	\\
\cline{2-10}     & \multicolumn{9}{c}{B. TE }\\ \cline{2-10} 
Bias	&	0.03		&	-0.04	&	 	&	-0.38	&	-0.51	&	-0.52	&	 &	&	 	\\
Std. Dev.	&	0.41		&	1.60	&	 	&	0.15	&	0.32	&	0.45	&	& 	&	 	\\
RMSE	&	0.41		&	1.60	&	 	&	0.41	&	0.60	&	0.69	&	 &	&	 	\\
Coverage	&	0.91		&	0.92	&	 	&	0.26	&	0.14	&	0.73	& 0.92 &	0.91	&	1.00	\\
Int. Length	&	1.41		&	5.73	&	 	&	0.62	&	0.60	&	1.83	& 56.53 &	2.28	&	6.84	\\
\cline{2-10}     & \multicolumn{9}{c}{C. PI }\\ \cline{2-10} 
Bias	&	0.04		&	0.34	&	 	&	-0.12	&	-0.08	&	-0.03	&	 &	&	 	\\
Std. Dev.	&	0.06		&	0.08	&	 	&	0.01	&	0.06	&	0.07	&	& 	&	 	\\
RMSE	&	0.07		&	0.35	&	 	&	0.12	&	0.10	&	0.07	&	 &	&	 	\\
Coverage	&	0.94		&	0.00	&	 	&	0.07	&	0.44	&	0.74	&	& 0.94	&	1.00	\\
Int. Length	&	0.25		&	0.29	&	 	&	0.02	&	0.12	&	0.19	& &	0.32	&	0.54	\\
\hline
\end{tabular*}													
\vspace{.1in}													
\end{table}		
\end{center}	
\center {\small \textsc{Fig Sim3.} Simulation Results: $n=400$, $p = 202$, $s_0 = 16$ } 
\smallskip
\hrule
\textcolor{white}{\footnotesize ,}
\hrule

\center \hspace{-1.5cm}\includegraphics[scale=.30]{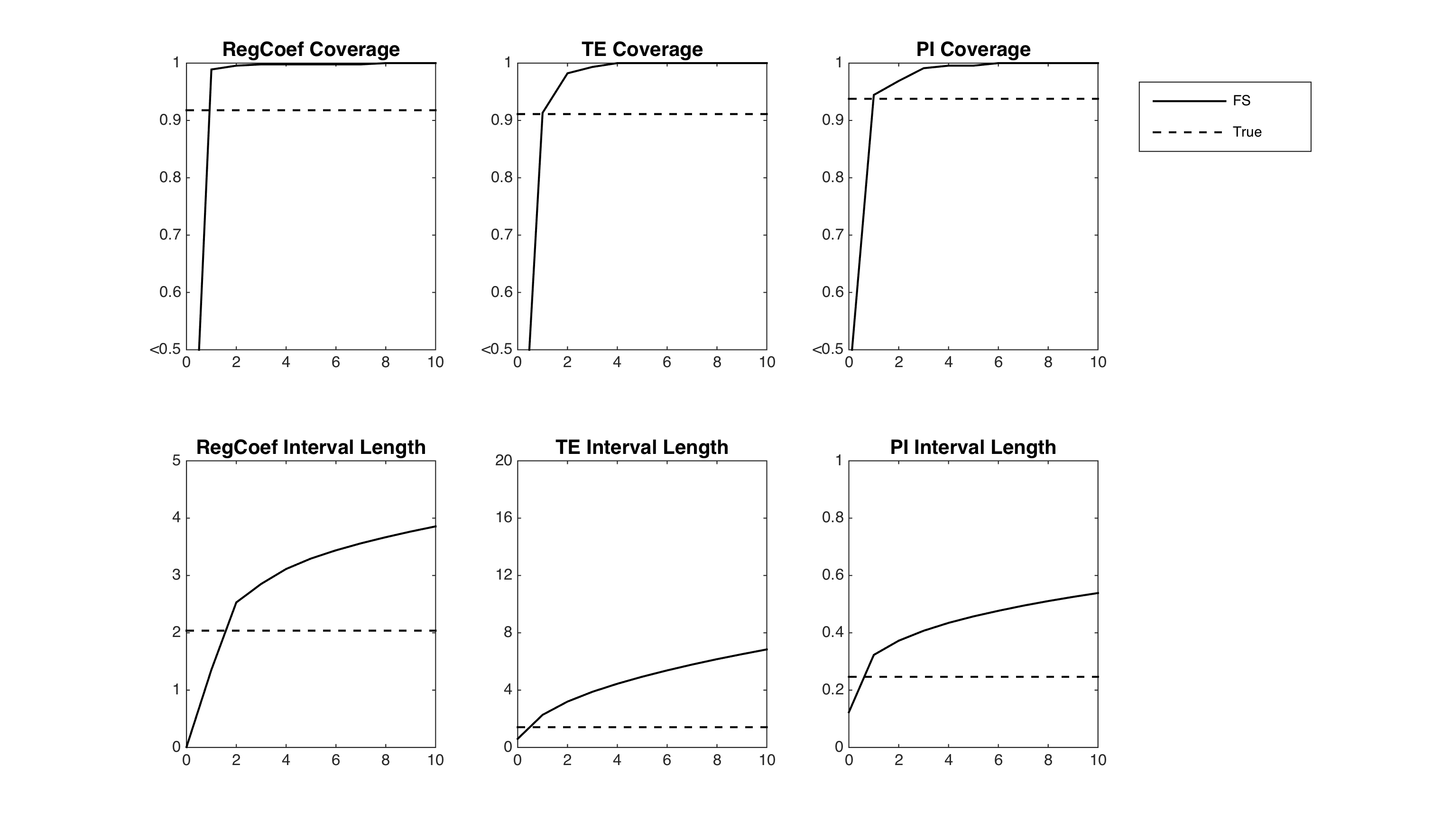}

\

\hrule

\begin{center} 
\begin{table}[H]   

\center {\small \textsc{Table Sim4.} Simulation Results: $n=400$, $p = 602$, $s_0 = 4$ } 

\vspace{.05in} 

\begin{tabular*}{\textwidth}{p{1.7cm} p{.9cm}  p{.8cm} p{.8cm} p{.8cm} p{.8cm} p{.8cm}  p{.8cm}   p{.8cm}  p{.8cm} }

\hline          \hline                                                                  \\ 
&	\ True \		& All		&	Double	&	Lasso	&	PL &LCV	& ZB&	TU(1)	&	TU(10) 	\\
\cline{2-10}     & \multicolumn{9}{c}{A. RegCoef }\\ \cline{2-10} 
Bias	&	-0.04		&	 	&	0.04	&	-0.12	&	-0.19	&	-0.44	&	 &	&	 	\\
Std. Dev.	&	0.69		&	 	&	0.64	&	0.12	&	0.38	&	0.47	&	 &	&	 	\\
RMSE	&	0.69		&	 	&	0.64	&	0.17	&	0.42	&	0.65	&	   &	&	 	\\
Coverage	&	0.92		&	 	&	0.92	&	0.12	&	0.08	&	0.42	& 0.76	&0.91	&	0.96	\\
Int. Length	&	2.43		&	 	&	2.22	&	0.23	&	0.28	&	1.09	&  1.01 &	1.86	&	4.21	\\
\cline{2-10}     & \multicolumn{9}{c}{B. TE }\\ \cline{2-10} 
Bias	&	-0.01		&	 	&	 	&	0.26	&	-0.03	&	-0.07	&	& 	&	 	\\
Std. Dev.	&	0.24		&	 	&	 	&	0.15	&	0.29	&	0.34	&  	& 	&	 	\\
RMSE	&	0.24		&	 	&	 	&	0.30	&	0.29	&	0.35	&   	& 	&	 	\\
Coverage	&	0.94		&	 	&	 	&	0.60	&	0.76	&	0.98	&  0.91 &	0.99	&	1.00	\\
Int. Length	&	0.87		&	 	&	 	&	0.67	&	0.63	&	1.86	&	2.21 & 2.12	&	7.92	\\
\cline{2-10}     & \multicolumn{9}{c}{C. PI }\\ \cline{2-10} 
Bias	&	0.00		&	 	&	 	&	-0.14	&	-0.02	&	-0.07	&	 & 	&	 	\\
Std. Dev.	&	0.07		&	 	&	 	&	0.01	&	0.08	&	0.06	&	& 	&	 	\\
RMSE	&	0.07		&	 	&	 	&	0.14	&	0.09	&	0.09	&	& 	&	 	\\
Coverage	&	0.94		&	 	&	 	&	0.04	&	0.77	&	0.72	&	& 0.92	&	1.00	\\
Int. Length	&	0.26		&	 	&	 	&	0.02	&	0.21	&	0.21	&	& 0.30	&	0.52	\\
\hline
\end{tabular*}													
\vspace{.1in}													
\end{table}		
\end{center}	

\center {\small \textsc{Fig Sim4.} Simulation Results: $n=400$, $p = 602$, $s_0 = 4$ } 
\smallskip
\hrule
\textcolor{white}{\footnotesize ,}
\hrule

\center  \hspace{-1.5cm}\includegraphics[scale=.30]{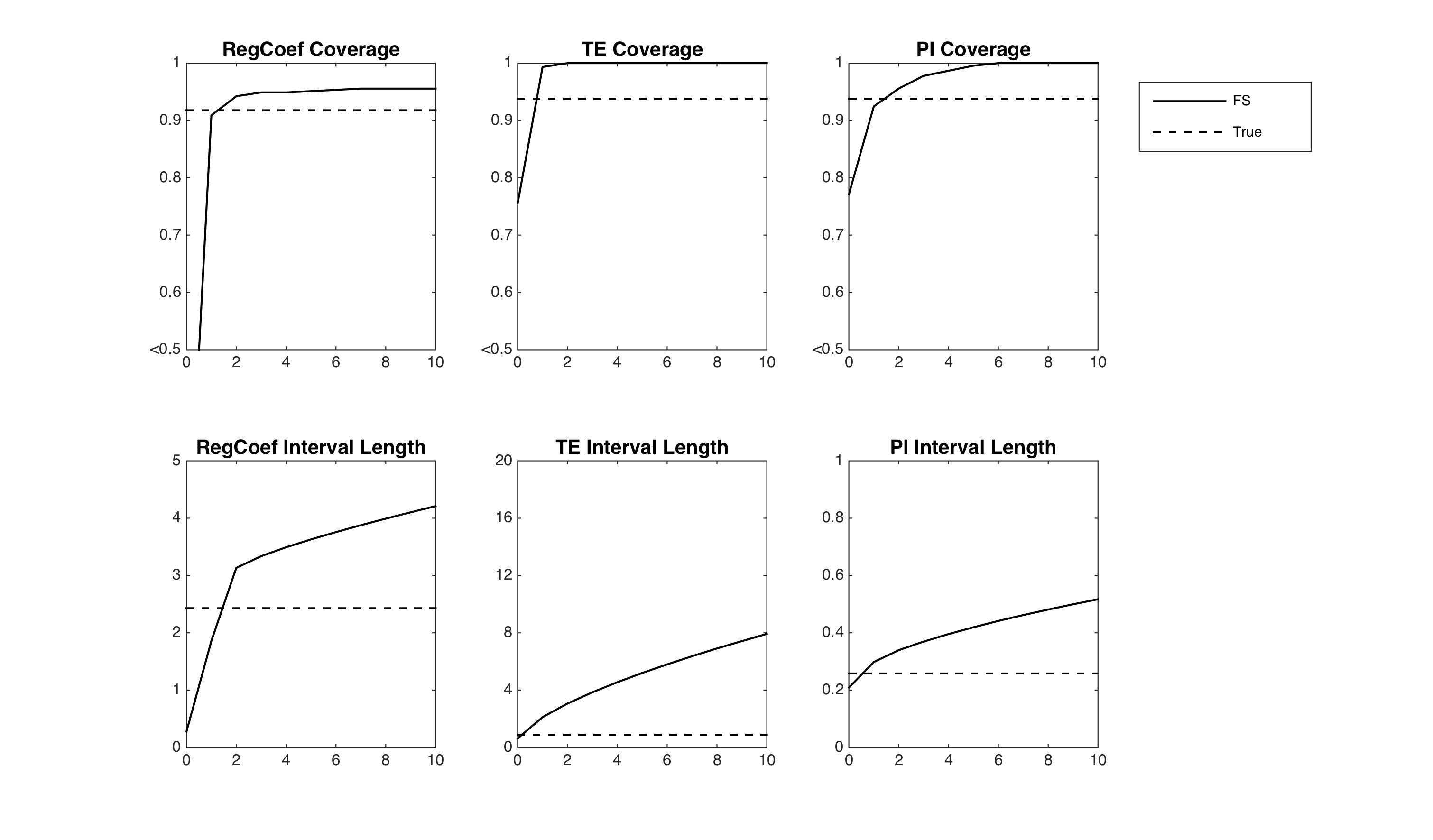}

\

\hrule

\begin{center} 
\begin{table}[H]   

{\small \textsc{Table Sim5.} Simulation Results: $n=400$, $p=602$, $s_0=8$  } 

\vspace{.05in} 

\begin{tabular*}{\textwidth}{p{1.7cm} p{.9cm}  p{.8cm} p{.8cm} p{.8cm} p{.8cm} p{.8cm}  p{.8cm}  p{.8cm}   p{.8cm} }

\hline          \hline                                                                  \\ 
&	\ True \		& All		&	Double	&	Lasso	&	PL & LCV	& ZB &	TU(1)	&	TU(10) 	\\
\cline{2-10}     & \multicolumn{9}{c}{A. RegCoef }\\ \cline{2-10} 
Bias	&	-0.03		&	 	&	0.78	&	-0.09	&	-0.08	&	-0.15	&	& 	&	 	\\
Std. Dev.	&	0.65		&	 	&	0.68	&	0.01	&	0.12	&	0.43	&	& 	&	 	\\
RMSE	&	0.65		&	 	&	1.03	&	0.09	&	0.15	&	0.45	&	& 	&	 	\\
Coverage	&	0.93		&	 	&	0.72	&	0.02	&	0.01	&	0.55	& 0.77	&1.00	&	1.00	\\
Int. Length	&	2.25		&	 	&	2.34	&	0.02	&	0.03	&	1.13	& 1.26	&2.04	&	4.68	\\
\cline{2-10}     & \multicolumn{9}{c}{B. TE }\\ \cline{2-10} 
Bias	&	-0.01		&	 	&	 	&	0.11	&	0.13	&	0.15	&	 &	&	 	\\
Std. Dev.	&	0.22		&	 	&	 	&	0.12	&	0.25	&	0.44	&	& 	&	 	\\
RMSE	&	0.22		&	 	&	 	&	0.17	&	0.28	&	0.46	&	& 	&	 	\\
Coverage	&	0.93		&	 	&	 	&	0.87	&	0.74	&	0.99	&	0.98 &1.00	&	1.00	\\
Int. Length	&	0.77		&	 	&	 	&	0.56	&	0.58	&	2.27	& 24.82 	&2.66	&	9.98	\\
\cline{2-10}     & \multicolumn{9}{c}{C. PI }\\ \cline{2-10} 
Bias	&	0.01		&	 	&	 	&	-0.09	&	-0.08	&	-0.04	&	 &	&	 	\\
Std. Dev.	&	0.10		&	 	&	 	&	0.11	&	0.11	&	0.11	&	& 	&	 	\\
RMSE	&	0.10		&	 	&	 	&	0.15	&	0.13	&	0.12	&	 &	&	 	\\
Coverage	&	0.95		&	 	&	 	&	0.85	&	0.88	&	0.89	&	&0.95	&	1.00	\\
Int. Length	&	0.40		&	 	&	 	&	0.44	&	0.43	&	0.40	&	 &0.51	&	0.85	\\
\hline
\end{tabular*}													
\vspace{.1in}													
\end{table}		
\end{center}	
\center {\small \textsc{Fig Sim5.} Simulation Results: $n=400$, $p = 602$, $s_0 = 8$ } 
\smallskip
\hrule
\textcolor{white}{\footnotesize ,}
\hrule

\center \hspace{-1.5cm}\includegraphics[scale=.30]{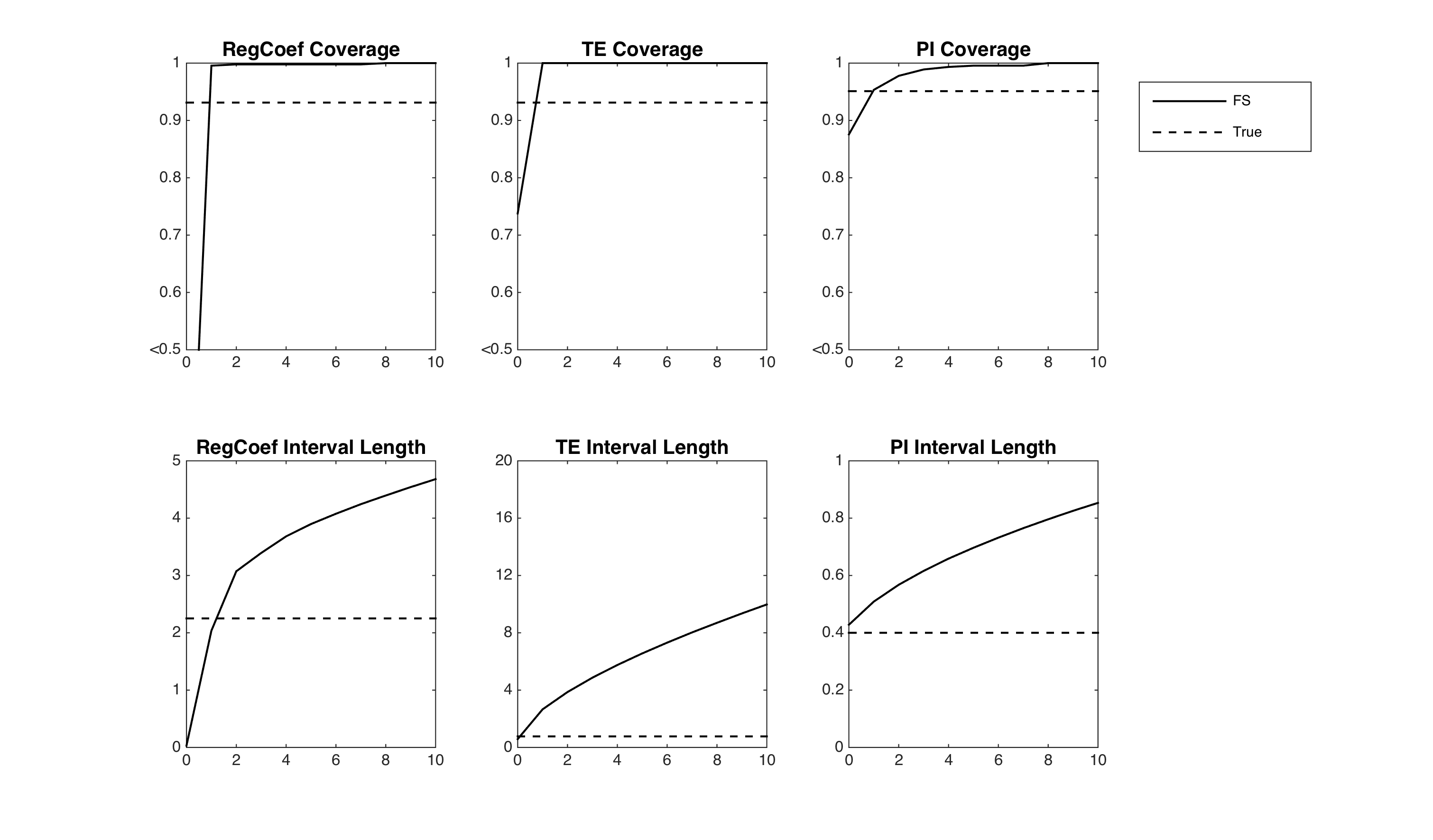}

\

\hrule

\begin{center}
\begin{table}[H]   
{\small \textsc{Table Sim6.} Simulation Results: $n=400$, $p=602$, $s_0=16$  } 

\vspace{.05in} 

\begin{tabular*}{\textwidth}{p{1.7cm} p{.9cm}  p{.8cm} p{.8cm} p{.8cm} p{.8cm} p{.8cm}  p{.8cm}  p{.8cm} p{.8cm} }
\hline          \hline                                                                  \\ 
&	\ True \		& All		&	Double	&	Lasso	&	PL & LCV	&	ZB & TU(1)	&TU(10) 	\\
\cline{2-10}     & \multicolumn{9}{c}{A. RegCoef }\\ \cline{2-10} 
Bias	&	-0.02		&	 	&	0.40	&	-0.06	&	-0.06	&	-0.10	&	 	&	 &	\\
Std. Dev.	&	0.55		&	 	&	0.54	&	0.00	&	0.00	&	0.22	&	 	&	& 	\\
RMSE	&	0.55		&	 	&	0.67	&	0.06	&	0.06	&	0.24	&	 	&	 &	\\
Coverage	&	0.94		&	 	&	0.88	&	0.00	&	0.00	&	0.35	&0.73 &	0.98	&	0.99	\\
Int. Length	&	2.01		&	 	&	2.04	&	0.00	&	0.00	&	0.50	&1.60&	1.33	&	4.15	\\
\cline{2-10}     & \multicolumn{9}{c}{B. TE }\\ \cline{2-10} 
Bias	&	0.00		&	 	&	 	&	-0.38	&	-0.51	&	-0.70	&	 &	&	 	\\
Std. Dev.	&	0.39		&	 	&	 	&	0.14	&	0.30	&	0.41	&	& 	&	 	\\
RMSE	&	0.39		&	 	&	 	&	0.41	&	0.60	&	0.81	&	 &	&	 	\\
Coverage	&	0.93		&		&	 	&	0.26	&	0.17	&	0.77	&	0.97 &0.94	&	1.00	\\
Int. Length	&	1.36		&	 	&	 	&	0.63	&	0.60	&	2.19	&	71.83 &2.76	&	9.97	\\
\cline{2-10}     & \multicolumn{9}{c}{C. PI }\\ \cline{2-10} 
Bias	&	0.04		&	 	&	 	&	-0.12	&	-0.08	&	-0.06	&	 &	&	 	\\
Std. Dev.	&	0.06		&	 	&	 	&	0.01	&	0.05	&	0.06	&	& 	&	 	\\
RMSE	&	0.07		&	 	&	 	&	0.12	&	0.10	&	0.08	&	 &	&	 	\\
Coverage	&	0.93		&		&	 	&	0.05	&	0.40	&	0.63	&	&0.93	&	1.00	\\
Int. Length	&	0.24		&	 	&	 	&	0.02	&	0.12	&	0.19	&	&0.33	&	0.61	\\
\hline
\end{tabular*}													
\vspace{.1in}														
\end{table}		
\end{center}	
\center {\small \textsc{Fig Sim6.} Simulation Results: $n=400$, $p = 602$, $s_0 = 16$ } 
\smallskip
\hrule
\textcolor{white}{\footnotesize ,}
\hrule

\center \hspace{-1.5cm}\includegraphics[scale=.30]{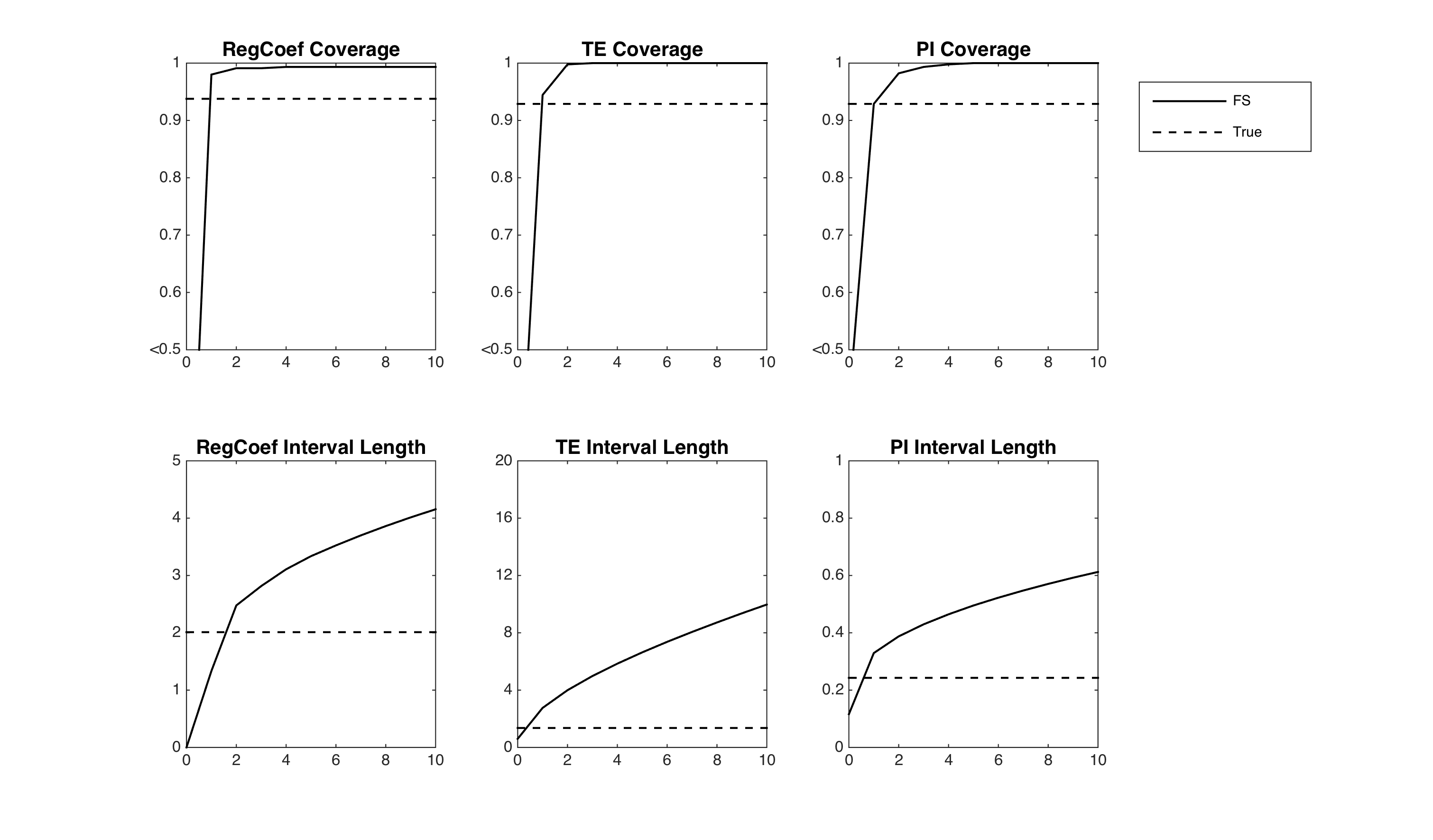}

\

\hrule

\end{document}